\documentclass[11pt,a4paper,reqno,english]{amsart}
\usepackage[T1]{fontenc}
\usepackage{amsmath}
\usepackage{amsthm}
\usepackage{amsfonts}
\usepackage{amssymb}
\usepackage{graphicx}
\usepackage{amsbsy}
\usepackage{eucal}
\usepackage{amsbsy}
\usepackage{mathrsfs}
\usepackage{bbm}

\newcommand{\vc}[1]{\overrightarrow{#1}}

\newcommand{\be}{\begin{equation}}
\newcommand{\ee}{\end{equation}}

\def\R{\mathbb{R}}

\def\e1{\vc{e}_1}
\def\e2{\vc{e}_2}
\def\e3{\vc{e}_3}
\def\dis{\displaystyle}

\catcode`@=11
\def\section{\@startsection{section}{1}%
  \z@{1.5\linespacing\@plus\linespacing}{.5\linespacing}%
  {\normalfont\bfseries\large\centering}}
\catcode`@=12

\newtheorem{theorem}{Theorem}[section]

\newtheorem{lemma}[theorem]{Lemma}

\newtheorem{proposition}[theorem]{Proposition}

\theoremstyle{remark}
\newtheorem{remark}[theorem]{\it \bf{Remark}\/}

\numberwithin{equation}{section}
\newcommand{\bea}{\begin{eqnarray}}
\newcommand{\eea}{\end{eqnarray}}
\newcommand{\bee}{\begin{eqnarray*}}
\newcommand{\eee}{\end{eqnarray*}}

\def\na{\nabla}

\def\e{\varepsilon}

\def\NN{\mathbb{N}}

\def\RR{\mathbb{R}}

\def\ds{\displaystyle}
\def\ni{\noindent}
\def\bs{\bigskip}
\def\ms{\medskip}

\def\eps{\varepsilon}

\def\bar#1{{\overline #1}}
\def\fref#1{{\rm (\ref{#1})}}
\def\pref#1{{\rm \ref{#1}}}

\def\calH{{\mathcal H}}

\def\calE{{\mathcal E}}

\catcode`@=11
\def\supess{\mathop{\operator@font Sup\,ess}}
\catcode`@=12
\def\un{{\mathbbmss{1}}}

\title[Ground states and blow-up for the Vlasov-Manev system]{Stable ground states and self-similar blow-up solutions for the gravitational Vlasov-Manev system}
\author{Mohammed Lemou}
\address{IRMAR and CNRS, Universit\'e de Rennes 1, France}
\email{mohammed.lemou@univ-rennes1.fr}
\urladdr{http://perso.univ-rennes1.fr/mohammed.lemou/}
\author{Florian M\'ehats}
\address{IRMAR, Universit\'e Rennes 1, France}
\email{florian.mehats@univ-rennes1.fr}
\urladdr{http://perso.univ-rennes1.fr/florian.mehats/}
\author{Cyril Rigault}
\address{IRMAR, Universit\'e Rennes 1, France}
\email{cyril.rigault@univ-rennes1.fr}

\begin{document}



\begin{abstract}
In this work, we study the orbital stability of steady states  and the existence of blow-up self-similar solutions to the so-called Vlasov-Manev (VM) system.
This system is a kinetic model which  has a similar Vlasov structure  as the classical Vlasov-Poisson system,  but is coupled to a potential in $-1/r- 1/r^2$ (Manev potential) instead of the usual gravitational potential in $-1/r$, and in particular the potential field does not satisfy a Poisson equation but a fractional-Laplacian equation.  We first prove the orbital stability of the ground states type solutions which are constructed as minimizers of the Hamiltonian, following the classical strategy: compactness of the minimizing sequences and the rigidity of the flow.  However, in driving this analysis, there are two mathematical obstacles: the first one is related to the possible blow-up of solutions to the VM system, which we overcome by imposing a sub-critical condition on the constraints of  the variational problem. The second difficulty (and the most important) is  related to the nature of the Euler-Lagrange equations (fractional-Laplacian equations) to which classical results for the Poisson equation do not extend. We overcome this difficulty  by proving  the uniqueness of the minimizer under equimeasurabilty constraints, using only the regularity of the potential and not the fractional-Laplacian Euler-Lagrange equations itself.  In the second part of this work,  we prove the existence of exact self-similar blow-up solutions to the Vlasov-Manev equation, with initial data arbitrarily close to ground states. This construction is based on a suitable variational problem with equimeasurability constraint.
\end{abstract}

\maketitle

\section{Introduction and main results}


In this paper, we study the stability of steady states and the existence of blow-up self-similar solutions to the Vlasov-Manev (VM) model for gravitational systems. In this mean field kinetic model, the usual Newtonian interaction potential is replaced by the so-called Manev potential. This potential corrects the Newtonian gravitational potential as follows:
$$U(x)=-\frac{1}{4\pi |x|}-\frac{\kappa}{2\pi^2 |x|^2},$$
where $\kappa$ is a positive constant. First it was studied by Manev  in the 1920' as an alternative way of Einstein's relativity  to explain the advance of the perihelion of Mercury unexplained by Newton's laws \cite{Ma1, Ma2, Ma3, Ma4}. And recently, F. Diacu, A. Mingarelli, V. Mioc and C. Stoica [6] followed by R. Illner, H.D. Victory, P.Dukes and A.V. Bobylev [2-3] gave the basics for the comeback of the Manev model, described by the first ones as "a fairly good substitute of relativity within the frame of classical mechanics".

We then consider  in this paper the case of a potentiel given by:
$$U(x)=-\frac{\delta}{4\pi |x|}-\frac{\kappa}{2\pi^2 |x|^2},$$
where $\delta$ is a nonnegative constant. Further physical studies of this potential can be found in  \cite{DM}. The case $\delta=0$, $\kappa=1$ will be referred to as the {\em pure Manev case}. The case $\delta> 0$, $\kappa\geq 0$  will be referred to as  {\em the Poisson-Manev case} which includes the Newtonian case $\delta=1$, $\kappa=0$. Note that at the limit $\delta\rightarrow1$ and $\kappa \rightarrow 0$, we recover the stability of steady states proved in \cite{LMR,LM2}.

Taking into account this correction, the standard Vlasov-Poisson system is replaced by the following Vlasov-Manev system:
\be\left\{\begin{array}{lc} \partial_{t}f+v \cdot \nabla_{x}f-\nabla_{x} \phi_f
\cdot \nabla_{v}f=0,   &(t,x,v) \in \R_{+}\times\R^{3}\times\R^{3},
\\ &  \\ f(t=0,x,v)=f_{0}(x,v) \geq 0,   & \\ \end{array}\right. \label{vm} \ee
in which $f=f(t,x,v)\geq 0$ is a distribution function  and $\phi_f$ the associated potential defined as follows. We have
\be
\phi_f(t,x)=\delta \phi_f^P+\kappa \phi_f^M, \label{pot}\ee
where $\phi_f^P$ and $\phi_f^M$ are respectively the Poisson potential and the Manev potential of $f$ given by:
\be  \phi_f^P(t,x)=-\int_{\R^3}\frac{\rho_f(t,y)}{4\pi |x-y|}dy,\qquad   \phi_f^M(t,x)=-
\int_{\R^3}\frac{\rho_f(t,y)}{2\pi^2|x-y|^{2}}dy, \label{2pot} \ee
 $\rho_f$ being the density associated with the distribution function $f$: 
$$\rho_f(t,x)=\int_{\R^3} f(t,x,v)dv.$$
Note that the two potentials satisfy
$$\triangle \phi_f^P= \rho_f \ \ \textrm{and}\ \ \left(-\triangle\right)^{1/2} \phi_f^M = -\rho_f, $$
and in particular the system \eqref{vm} reduces to the well-known gravitational Vlasov-Poisson system in the case $\delta=1$ and $\kappa=0$. 

To our knowledge, the only existing mathematical analysis of the Vlasov-Manev model is due to Bobylev, Dukes, Illner and Victory   \cite{VM1, VM2}. In these works, the local existence of regular solutions is proved and some questions of global existence and finite-time blow-up are discussed.

We now  give some basic properties of the Vlasov-Manev system \eqref{vm}. Sufficiently regular solutions  to \eqref{vm} on  a time interval $[0,T]$ satisfy the conservation of the so-called Casimir functionals:
\be \forall t\in [0,T],\qquad \ \|j(f(t))\|_{L^1}=\|j(f_0)\|_{L^1} \label{conserv-j}\ee
and the conservation of the Hamiltonian
$$ \forall t\in [0,T], \qquad  \ \mathcal{H}(f(t)) = \mathcal{H}(f_0), $$
where $j$ is any smooth real-valued function with $j(0)=0$, and where
\be \mathcal{H}(f(t))=\left\| |v|^{2}f(t)\right\|_{L^1}-E_{pot}(f(t)). \label{def-H}\ee
The potential energy $E_{pot}$  is defined by  
$$
E_{pot}(f(t))=-\dis \int_{\R^3}\phi_f(t,x)\rho_f(t,x)dx=\delta E_{pot}^P(f(t))+\kappa E_{pot}^M(f(t)),
$$
where we have denoted
$$
E_{pot}^P(f(t))=- \dis \int_{\R^3}\phi_f^P(t,x)\rho_f(t,x)dx\quad\mbox{and}\quad E_{pot}^M(f(t))=- \dis \int_{\R^3}\phi_f^M(t,x)\rho_f(t,x)dx.
$$
These potential energies are controlled thanks to standard interpolation inequalities:
\be
0\leq E_{pot}^P(f)\leq C_1\left\| |v|^{2}f\right\|_{L^1}^{\frac{1}{2}}
\left\|f\right\|_{L^1}^{\frac{7p-9}{6(p-1)}}\left\| f \right\|_{L^p}^{\frac{p}{3(p-1)}},\label{ineg-inter-P}\ee
\be 
0\leq E_{pot}^M(f)\leq C_2\left\| |v|^{2}f\right\|_{L^1}
\left\|f\right\|_{L^1}^{\frac{p-3}{3(p-1)}}\left\|f\right\|_{L^p}^{\frac{2p}{3(p-1)}}, \label{ineg-inter-M}\ee
for all $p\geq 3$.

Our aim in  this paper is twofold. First we  prove the  orbital stability of ground states type  stationary solutions to  the Vlasov-Manev problem.  Second we prove the existence of exact self-similar solutions to the pure Manev case and in particular we construct a continuous family of blow-up solutions to this system around each ground state. While the question of non linear stability has not been studied in the past for the VM system, it has attracted  considerable attention in the case of the Vlasov-Poisson system ($\kappa=0$), both in physics (see \cite{A1,A2},  \cite{binney} and the references therein) and mathematics community  \cite{Wol,G,GR1,GR2,LMR,SS,LM3}.  
We emphasize that the structure of the equation in the pure Manev case ($\delta=0$, $\kappa=1$) can be compared in some sense with the Vlasov-Poisson system in dimension 4 (see \cite{LMR}), where blow-up self-similar profiles and pseudo-conformal symmetry are exhibited. In the pure Manev case, we shall construct ground states by minimizing the constant in the interpolation inequality \fref{ineg-inter-M}, following the standard strategy as in the case of nonlinear Schr\"odinger equation \cite{weinstein}.

On the other hand, as already noticed in \cite{VM1}, the case of the general VM system ($\delta>0$) shares similar mathematical properties with the relativistic  Vlasov-Poisson system \cite{LM2}. In \cite{LM2}, the stability of steady state solutions to the relativistic Vlasov-Poisson equation is proved by minimizing the energy and by using a homogeneity-breaking property which comes from the  fact  that the relativistic kinetic energy is a non-homogeneous velocity moment of the distribution function. In the present case of VM system, the homogeneity-breaking  comes from  the presence of two contributions in the general VM potential with different homogenities.  This homogeneity-breaking property makes possible to build a well-posed  variational problem provided a sub-critical condition is imposed on the constraints. Notice that the subcritical condition for the well-posedness of the variational problem in the context of the relativistic Vlasov-Poisson system was also observed in \cite{KTZ}. In driving the classical approach in a similar way as in \cite{LM2} and  \cite{LMR}, a new important difficulty appears. This difficulty is related to the nature of the Euler-Lagrange equations  
to which classical results for the Poisson equation do not extend.  In the classical VP case, a complete stability result is generally obtained by using both the Euler-Lagrange  equation (which is equivalent to a non linear Poisson equation) and the rigidity of the flow.  In the present case, the Euler-Lagrange equation is a fractional-Laplacian equation, and this prevents from using ODE techniques. Nevertheless, we prove the uniqueness of the minimizer under equimeasurable constraints by a new argument  which completely avoids ODE techniques. This argument is again used, together with the help of suitable rearrangement techniques as introduced in \cite{LMR-inv, LM3}, to prove the existence of exact self-similar solutions in the pure Manev case, and  to build a continuous family of blow-up solutions around each minimizer.

\bs

In order to state our main results, let us make precise our assumptions. Consider a function $j:\R_{+}\rightarrow\R_{+}$ satisfying the  following hypotheses.\\
(H1) $j$ is a $\mathcal{C}^2$ function, with $j(0)=j'(0)=0$ and such that $j''(t)>0$ for $t>0$.\\
(H2) There exist $p, q >3$ such that
\be p\leq \dfrac{tj'(t)}{j(t)}\leq q,\ \ \ \ \forall t>0. \label{j-H3}\ee 
\\
We  note that (H2) is equivalent to the nondichotomy condition:
\be b^{p}j(t) \leq j(bt) \leq b^{q}j(t),\ \ \forall b\geq 1,\ t\geq 0. \label{j-nondicho}\ee 
For a function $j$ satisfying (H1) and (H2), we define the corresponding energy space
\be \mathcal{E}_j=\{f\geq 0\ \textrm{ such that }\|f\|_{\mathcal{E}_j}:=\ \|f\|_{L^1}+\|j(f)\|_{L^1}+\left\| |v|^2 f\right\|_{L^1} < +\infty\}
\label{def-espace}\ee
and we shall say that a sequence $f_n$ converges to $f$ in $\mathcal{E}_j$ if
$$\|f_n-f\|_{L^1}\to 0,\quad \|j(f_n-f)\|_{L^1}\to 0\quad \mbox{and}\quad \left\| |v|^2 (f_n-f)\right\|_{L^1}\to 0.$$
{}From the interpolation inequality \eqref{ineg-inter-M}, the following constant is strictly positive:
\be K_j^M=\inf_{f \in \mathcal{E}_j\backslash\{0\}}K_j^M(f) \ \ \textrm{with} \ \ K_j^M(f)=\frac{\left\| |v|^{2}f\right\|_{L^1}
\left\|f\right\|_{L^1}^{\frac{p-3}{3(p-1)}}\left\| f+j(f) \right\|_{L^1}^{\frac{2}{3(p-1)}}}{E_{pot}^M(f)}.\label{def-KM}\ee
Indeed, from \fref{j-nondicho} one has $t+j(t)\geq Ct^p$ for all $t\geq 0$.

\bs
 In our first result, we establish the existence of ground states for the Vlasov-Manev problem.
\begin{theorem} [Existence of ground states]\label{theo1}
Let $j$ be a function satisfying {\rm (H1)} and {\rm (H2)}. \\
(i) {\rm Poisson-Manev case} ($\delta > 0$). Let $M_1>0$, $M_j>0$ such that
\be \kappa M_1^{\frac{p-3}{3(p-1)}}(M_1+M_j)^{\frac{2}{3(p-1)}}<K_j^M, \label{subcrit}\ee
where $K_j^M$ is defined by \eqref{def-KM}, and let 
$$\mathcal{F}(M_1,M_j)=\{f\in \mathcal{E}_j,\ \|f\|_{L^1}=M_1,\ \|j(f)\|_{L^1}=M_j\}.$$ 
Then there exists a steady state of \eqref{vm} which minimizes the variational problem
\be I(M_1,M_j)= \inf_{f\in \mathcal{F}(M_1,M_j)}\mathcal{H}(f),\label{def-I}\ee
where $\mathcal{H}$ is the Hamiltonian defined by \eqref{def-H}.\\
(ii) {\rm Pure Manev case ($\delta=0$, $\kappa=1$)}. For all $M_1, M_j>0$, the following variational problem
\be J(M_1,M_j)= \inf_{f\in \mathcal{F}(M_1, M_j)}K(f),\quad \mbox{with }K(f):=\frac{\left\| |v|^{2}f\right\|_{L^1}
}{E_{pot}^M(f)}.\label{def-K}\ee
admits a minimizer. Furthermore, for any given $M_1>0$, there exists a unique $M_j>0$ such that $J(M_1,M_j)=1$. Moreover, the minimizers of \fref{def-K} are steady states to \fref{vm} if, and only if $J(M_1,M_j)=1$.\\
(iii) In both cases ($\delta\geq 0$), any steady state $Q$ obtained as a minimizer of \fref{def-I} or \fref{def-K} is continuous, compactly supported and takes the form 
\be Q(x,v)=(j')^{-1}\left(\dfrac{\frac{|v|^2}{2}+\phi_{Q}(x)-\lambda}{\mu}\right)_+ \label{expression-Q}\ee
where $\lambda$ and $\mu$ are negative constants. Moreover, $\phi_{Q}(x)$ is spherically symmetric (up to a translation shift), increasing  and belongs to $\mathcal{C}^{1,\alpha}$, for all $\alpha\in(0,1)$.
In \eqref{expression-Q}, we used the notation $a_+= max(a,0)$.
\end{theorem} 
Notice that in the case $\delta=1$ and $\kappa=0$, the condition \eqref{subcrit} is always satisfied. In this case, the Vlasov-Manev system \eqref{vm} is nothing but the classical Vlasov-Poisson system, for which it is already known that minimizers of the two constraints problem \eqref{def-I} always exist and that the minimizing sequences are compact, see \cite{LMR}. In \cite{LM2}, the orbital stability in the case of the VP system has been proved thanks to a uniqueness result of these minimizers which was based on a  combination of  the Poisson equation and the rigidity of the flow.

\bs
Our second main result concerns the orbital stability  of the above constructed ground states under the action of the Vlasov-Manev flow. As in  \cite{LM2}, the proof of these stability results needs in a crucial way the uniqueness of the minimizer under some flow constraints (namely the equimeasurability property).  However in  \cite{LM2},  the proof of this uniqueness was based on the use of the Poisson equation satisfied by the minimizer. Here, the Euler-Lagrange equation is a fractional Laplacian equation and the proof of  \cite{LM2} cannot be used. In fact, we prove this uniqueness result in a way that completely avoids the use of the Euler-Lagrange equation, and in particular, this generalizes also the uniqueness result obtained in   \cite{LM2}. The only property of the minimizers that we use is their equimeasurablity. In particular, our proof avoids the usual ODE techniques, which in fact, are useless here since the Euler-Lagrange equation is a fractional-Laplacian equation. 

\begin{lemma}[Uniqueness of the minimizer under equimeasurability condition]\label{theo2}
Let $F \in \mathcal{C}^0(\R,\R_+)$, strictly decreasing on $\R_-$, such that $F(\R_-)=\R_+$ and $F(\R_+)=\{0\}$. We define 
$$Q_0(x,v)=F\left(\frac{|v|^2}{2}+\psi_0(x)\right), \qquad Q_1(x,v)=F\left(\frac{|v|^2}{2}+\psi_1(x)\right)$$
 on $\R^3\times \R^3$, where $\psi_0$ and $\psi_1$ are two nondecreasing continuous radially symmetric potentials such that the sets $\{x\in \R^3,\ \psi_0(x)<0\}$ and $\{x\in \R^3,\ \psi_1(x)<0\}$ are bounded. Then the equimeasurability of $Q_0$ and $Q_1$ for the Lebesgue measure in $\R^6$, i.e.
\be \forall t>0,\ \mbox{meas}\{(x,v)\in \R^6,\ Q_0(x,v)>t\}= \mbox{meas}\{(x,v)\in \R^6,\ Q_1(x,v)>t\},\label{def-equi}\ee
implies that $Q_0=Q_1$. In particular:\\
(i) {\rm Poisson-Manev case} ($\delta \neq 0$):  two equimeasurable steady states of \eqref{vm} which minimize \eqref{def-I} under the subcritical condition \eqref{subcrit} are equal up to a translation in space.\\
(ii) {\rm Pure Manev case} ($\delta = 0$, $\kappa=1$): two equimeasurable steady states of \eqref{vm} which minimize \fref{def-K} and which have the same kinetic energy are equal up to a translation shift in space.
\end{lemma}

Now, using the compactness of all the minimizing sequences  of \fref{def-I} and \fref{def-K} (which will be proved) and the uniqueness result stated in Lemma \ref{theo2} we get the desired stability results. 
\begin{theorem}[Orbital stability of ground states]\label{theo3}\mbox{}\\
(i) {\rm Poisson-Manev case} ($\delta>0$). Let $M_1, M_j>0$ satisfy the subcritical condition \eqref{subcrit}. Then any steady state $Q$ of \eqref{vm} which minimizes \eqref{def-I} is orbitally stable under the flow \eqref{vm}. More precisely, given $\varepsilon>0$, there exists $\delta(\varepsilon)>0$ such that the following holds true. Consider $f_0$ a smooth function with $\|f_0-Q\|_{\mathcal{E}_j}\leq \delta(\varepsilon)$,
and let $f(t)$ be a classical solution to \eqref{vm} on a time interval $[0,T)$, $0<T\leq +\infty$, with initial data $f_0$. Then there exists a translation shift $x(t)\in \R^3$ such that, for all $t\in [0,T)$, we have
$$\left\|f(t,x+x(t),v)-Q\right\|_{\mathcal{E}_j} <\varepsilon.$$
(ii) {\rm Pure Manev case} ($\delta=0$, $\kappa=1$). Let $Q$ be a steady state of \eqref{vm} which minimizes \eqref{def-K}. Then for all $\eps>0$, there exists a constant $\delta(\varepsilon)>0$ such that the following property holds true. Let $f(t)$ be a classical solution to \eqref{vm} on a time interval $[0,T)$, $0<T\leq +\infty$, with initial data $f_0$, satisfying:\\
\hspace*{5mm}(a) $\|f_0-Q\|_{L^1} \leq \delta(\varepsilon)$ and $\|j(f_0)\|_{L^1} \leq \|j(Q)\|_{L^1} + \delta(\varepsilon)$,\\
\hspace*{5mm}(b) $\forall t\in[0,T)$, $\quad \lambda(t)^2 \mathcal{H}(f(t)) < \delta(\varepsilon)\quad $ where $\lambda(t)=\left(\frac{ \| |v|^2 Q \|_{L^1}}{ \| |v|^2 f(t) \|_{L^1}}\right)^{1/2}$.\\
Then there exists a translation shift $x(t)\in \R^3$ such that, for all $t\in [0,T)$, we have
$$ \left\| f\left(t,\lambda(t)(x+x(t)),\frac{v}{\lambda(t)}\right)-Q \right\|_{\mathcal{E}_j} <\varepsilon.$$
\end{theorem}

\noindent
\begin{remark} The goal here is to prove  this stability result assuming the framework of classical solutions to the Vlasov-Manev model, and not to solve the Cauchy problem. For smooth initial data decaying fast enough at the infinity, the local existence and the uniqueness of  regular solutions to \eqref{vm} has been proved in \cite{VM2}. The global existence of classical solutions is an open problem. Our result shows that the solutions remain in the vicinity of the ground state $Q$ (up to a translation shift), but does not a priori exclude a possible blow-up of some derivative of $f$. 
\end{remark}

Notice that one may not be aimed  at a better stability than the blow-up stability in the pure Manev case. Indeed the classical stability does not hold as shown  by the following example (translating the pseudo-conformal symmetry property in this case): let $g=g(x,v)$ be a steady state of \eqref{vm} in the pure Manev case ($\delta=0$, $\kappa=1$), then the function $f_T$ defined by
$$f_T(t,x,v)=g \left( \frac{Tx}{T-t},\frac{T-t}{T}v+\frac{x}{T}\right),$$
is a blow-up solution to the system \eqref{vm} in the pure Manev case ($\delta=0$, $\kappa=1$), see \cite{VM1,VM2}. 

\bs
To go further with the pure Manev system our last result gives the existence of exact spherically symmetric self-similar solutions (we recall that a spherically symmetric function, in this context, is a function which only depends of $|x|$, $|v|$ and $x\cdot v$).

\begin{theorem}[Exact self-similar solutions in the pure Manev case]\label{thm6}  Let $Q$ be a steady state of \eqref{vm} in the pure Manev case ($\delta=0$, $\kappa=1$), which minimizes \eqref{def-K}.
Then there exists a constant $b^\ast>0$ such that for all $b\in[0,b^\ast]$, there exists a compactly supported spherically symmetric stationary profile $Q_b\in \mathcal{C}^0(\R^6)$ having the form
$$
Q_b(x,v)=F_b\left(  \frac{|v|^2}{2}+b x\cdot v + \phi_{Q_b}(x)\right)
$$
on its support, and such that, for all $T>0$, the function
\be f(t,x,v)=Q_b\left(\frac{x}{\lambda(t)},\lambda(t)v\right) \ \ \textrm{with} \ \lambda(t)=\sqrt{2b(T-t)} \label{selfsim2}\ee 
is an exact self-similar blow-up solution to the pure Manev system \eqref{vm} in $\mathcal{E}_j$. Here, the function $\phi_{Q_b}$ belongs to $\mathcal C^1$ and the function $F$ is a continuous nonnegative function on $\RR$, which is $\mathcal C^1$ on $]-\infty,e_b[$ for some $e_b<0$ and vanishes on $[e_b,+\infty[$. Moreover,  $Q_b$ converges to $Q_0=Q$ in $\mathcal{E}_j$ as $b\rightarrow 0$.
\end{theorem}
\bs
\begin{remark}
The previous results show that, in the pure Manev case ($\delta=0$, $\kappa=1$), around the ground state $Q$, there are at least three classes of dynamical profiles.\\
(i) {\em Subcritical solutions.} When the initial data $f_0$ is subcritical, i.e. when
$$J\left(\|f_0\|_{L^1},\|j(f_0)\|_{L^1}\right)>1=J\left(\|Q\|_{L^1},\|j(Q)\|_{L^1}\right),$$
then the kinetic energy of the solution $f(t)$ is controlled for all time. Indeed, one has
$$\calH(f(t))\geq \||v|^2f\|_{L^1}\left(1-\frac{1}{K(f(t))}\right)\geq \||v|^2f\|_{L^1}\left(1-\frac{1}{J\left(\|f_0\|_{L^1},\|j(f_0)\|_{L^1}\right)}\right).$$
Recall that $J$ is defined by \fref{def-K} and is continuous and decreasing with respect to its two  arguments.
\\
(ii) {\em Pseudo-conformal blow-up solutions.} The following family gives an explicit class of finite time blow-up solutions \cite{VM1,VM2}:
$$f(t,x,v)=Q\left( \frac{Tx}{T-t},\frac{T-t}{T}v+\frac{x}{T}\right),\qquad T>0.$$
Note that the kinetic energy blows up with the rate $(T-t)^{-2}.$
\\
(iii) {\em Self-similar blow-up solutions.} The family given by \fref{selfsim2} blows up in finite time and the kinetic energy blows up with the rate $(T-t)^{-1}.$
\end{remark}
The outline of the paper is as follows. Section \ref{sect1} deals with the proof of Theorem \ref{theo1}. After preliminary technical results concerning some properties of the infimum $I(M_1,M_j)$  (Subsection \ref{subsect1}), we prove in Subsection \ref{subsect2} the existence of minimizers. Then we characterize the ground states: Euler-Lagrange equation, regularity and spherical symmetry. Section \ref{sect2} is devoted to the proof of stability of the ground state through the Vlasov-Manev flow as stated in Theorem \ref{theo3}. First, in  Subsection \ref{subsect3}, we prove the uniqueness of the ground state in the class of equimeasurable functions, Lemma \ref{theo2}. Then we use standard concentration-compactness arguments to prove the compactness of minimizing sequences. Combining the uniqueness and compactness properties, we finally deduce the orbital stability result, Theorem \ref{theo3}. Section \ref{sectselfsim} is devoted to the proof of the Theorem \ref{thm6}: in Subsection \ref{subsectrearrang} we introduce the rearrangement with respect of a modified microscopic energy and apply it in Subsection \ref{subsectselfsim} to build self-similar solutions of \eqref{vm} in the pure Manev case $\delta=0$, $\kappa=1$.

\section{Existence of ground states}
\label{sect1}
This section is devoted to the proof of Theorem \ref{theo1}.
\subsection{Properties of the infimum}
\label{subsect1}

In this section, we prove two lemmas concerning some monotonicity properties of the infimum defined by \eqref{def-I} and by \eqref{def-K}.
\begin{lemma} [Monotonicity properties of the infimum $I(M_1,M_j)$]
\label{lemmonotone}
Let $j$ be a real-valued function  satisfying Assumptions {\rm (H1)} and {\rm (H2)}, let $M_1>0$ and $M_j>0$ such that \eqref{subcrit} holds, and let $I(M_1,M_j)$ be defined by \eqref{def-I} in the case $\delta>0$. Then we have
\be -\infty < I(M_1,M_j) <0. \label{encad-I} \ee
Moreover the following nondichotomy condition holds true: for all $0<\alpha <1$ and $0\leq \beta \leq 1$,
\be I(\alpha M_1,\beta M_j)+I((1-\alpha)M_1,(1-\beta)M_j)>I(M_1,M_j). \label{nondicho-I}\ee
\end{lemma}
\begin{proof}[Proof. ]

\textsl{Step 1. The infimum is finite and negative.}

\bs
\ni
We first prove \eqref{encad-I}. Let $f\in \mathcal{F}(M_1,M_j)$. Then from \eqref{ineg-inter-P} and \eqref{def-KM}, we have
\be \begin{array}{cl}
\mathcal{H}(f) & \dis \geq \left\| |v|^{2}f\right\|_{L^1} -\frac{\kappa}{K^M_j}\left\| |v|^{2}f\right\|_{L^1}
M_1^{\frac{p-3}{3(p-1)}}(M_1+M_j)^{\frac{2}{3(p-1)}}\\
&\dis \hspace*{2.1cm} - C\delta \left\| |v|^{2}f\right\|_{L^1}^{1/2}
M_1^{\frac{7p-9}{6(p-1)}}(M_1+M_j)^{\frac{1}{3(p-1)}} \\ 
 &\dis\hspace*{-1cm} \geq \left\| |v|^{2}f\right\|_{L^1} \left(1-\frac{\kappa}{K^M_j}M_1^{\frac{p-3}{3(p-1)}}(M_1+M_j)^{\frac{2}{3(p-1)}}\right)- C_{M_1,M_j} \left\| |v|^{2}f\right\|_{L^1}^{1/2} \label{controlkinetic}.
\end{array}
\ee 
Now the  subcritical condition \eqref{subcrit} implies that $$1-\frac{\kappa}{K^M_j}M_1^{\frac{p-3}{3(p-1)}}(M_1+M_j)^{\frac{2}{3(p-1)}}>0.$$
Thus $\mathcal{H}(f)$ is bounded from below, which proves that $I(M_1,M_j)$ is finite. To prove that $I(M_1,M_j)$ is negative, we use a rescaling argument. For $\lambda>0$ and $f\in \mathcal{F}(M_1,M_j)$, we consider the rescaled function $\tilde{f}(x,v)=f(\frac{x}{\lambda},\lambda v)$. Then $\tilde{f}$ belongs to $\mathcal{F}(M_1,M_j)$ and we have (see Appendix A)
$$\begin{array}{ccl}
\mathcal{H}(\tilde{f}) & = & \dis \frac{1}{\lambda^2}\left\| |v|^{2}f\right\|_{L^1} - \frac{\delta}{\lambda}  E_{pot}^P(f)-\frac{\kappa}{\lambda^2} E_{pot}^M(f) \\ \\
 & \sim & \dis - \frac{1}{\lambda} E_{pot}^P(f)\ \ \textrm{as}\ \lambda \rightarrow +\infty,
\end{array}$$  
where $E_{pot}^P(f)$ is positive (since $f$ is not zero). The property \eqref{encad-I} follows.\\

\noindent
\textsl{Step 2. The nondichotomy condition.}
\nopagebreak

\bs
\ni
We now claim the following monotonicity properties: for all $0<k\leq 1$,
\be I(M_1,kM_j) \geq k^{\frac{1}{3(q-1)}}I(M_1,M_j) \label{I-supj}\ee
and
\be I(kM_1,M_j) \geq k^{\frac{4p-6}{3(p-1)}}I(M_1,M_j). \label{I-sup1}\ee 
{\em Proof of \eqref{I-supj}.} Let $k \in (0,1]$ and $f\in \mathcal{F}(M_1,kM_j)$. From Appendix \ref{appA}, consider the unique rescaled function $\tilde{f}(x,v)=\alpha f(\alpha^{1/3}x,v)$ in $ \mathcal{F}(M_1,M_j)$. From \eqref{lambda-gamma}, we deduce in particular that $\alpha\geq 1$ and that
$$\alpha^{p-1}\leq \frac{1}{k}\leq \alpha^{q-1}.$$
Then, we get
$$ \mathcal{H}(\tilde{f}) = \left\| |v|^{2}f\right\|_{L^1} - \alpha^{\frac{1}{3}} \delta E_{pot}^P(f) - \alpha^{\frac{2}{3}}\kappa E_{pot}^M(f) \leq \left\| |v|^{2}f\right\|_{L^1} - \alpha^{\frac{1}{3}}\left(\delta E_{pot}^P(f) + \kappa E_{pot}^M(f)\right)$$
and
$$ I(M_1,M_j)\leq\mathcal{H}(\tilde{f}) \leq \left\| |v|^{2}f\right\|_{L^1} - \left(\frac{1}{k}\right)^{\frac{1}{3(q-1)}}\left(\delta E_{pot}^P(f) + \kappa E_{pot}^M(f)\right) \leq  \left(\frac{1}{k}\right)^{\frac{1}{3(q-1)}}\mathcal{H}(f).$$
This result holds for all $f\in \mathcal{F}(M_1,kM_j)$ and $k \in (0,1]$, which proves \eqref{I-supj}.

\ms
\ni
{\em Proof of \eqref{I-sup1}.} Similarly, we take $f\in\mathcal{F}(kM_1,M_j)$ and set 
$$\tilde{f}(x,v)=\alpha f(\alpha^{1/3}k^{1/3}x,v)$$
the unique rescaled function in  $\mathcal{F}(M_1,M_j)$, which implies that $\alpha\leq 1$ with
$$\alpha^{q-1}\leq k\leq \alpha^{p-1}.$$
Thus, 
$$ \begin{array}{ccl}
\mathcal{H}(\tilde{f})& = & \dis \frac{1}{k}\left\| |v|^{2}f\right\|_{L^1} - \alpha^{\frac{1}{3}}k^{-5/3}\delta E_{pot}^P(f) - \alpha^{\frac{2}{3}}k^{-4/3}\kappa E_{pot}^M(f) \\
 & \leq & \dis \frac{1}{k}\left\| |v|^{2}f\right\|_{L^1} - \left(\frac{1}{k}\right)^{\frac{5p-6}{3(p-1)}}\delta E_{pot}^P(f) -\left(\frac{1}{k}\right)^{\frac{4p-6}{3(p-1)}} \kappa E_{pot}^M(f).
\end{array} $$
{}From $k\leq1$ and $p>3$, we conclude that
$$ I(M_1,M_j)\leq \mathcal{H}(\tilde{f}) \leq \left(\frac{1}{k}\right)^{\frac{4p-6}{3(p-1)}}\mathcal{H}(f),$$
and \eqref{I-sup1} follows.

We now prove \eqref{nondicho-I}. Let $0<\alpha <1$ and $0\leq \beta \leq 1$. Then \eqref{I-supj} and \eqref{I-sup1} imply
$$I(\alpha M_1,\beta M_j)\geq \alpha^{\frac{4p-6}{3(p-1)}}\beta^{\frac{1}{3(q-1)}}I(M_1,M_j),$$
and a similar inequality with $(1-\alpha)$ and $(1-\beta)$. As we have $I(M_1,M_j)<0$, we only have to show that
$$ \alpha^{\frac{4p-6}{3(p-1)}}\beta^{\frac{1}{3(q-1)}}+(1-\alpha)^{\frac{4p-6}{3(p-1)}}(1-\beta)^{\frac{1}{3(q-1)}}<1,$$
which holds true since $q>1$ and $\frac{4p-6}{3(p-1)}>1$. The proof of Lemma \ref{lemmonotone} is complete.
\end{proof} 

Now we state the second lemma concerning the pure Manev case.
\begin{lemma} [Monotonicity properties of the infimum $J(M_1,M_j)$]
\label{lemmonotone-K}
Let $j$ be a real-valued function  satisfying Assumptions {\rm (H1)} and {\rm (H2)}, let $M_1, M_j>0$ and let $J(M_1,M_j)$ be defined by \eqref{def-K} in the case $\delta=0$. Then for all $0<\alpha\leq 1$ we have
\be  \alpha^{\frac{3-p}{3(p-1)}}J(M_1, M_j)\leq J(\alpha M_1,M_j) \leq  \alpha^{\frac{3-q}{3(q-1)}}J(M_1,M_j), \label{control-KB1} \ee
\be  \alpha^{-\frac{2}{3(q-1)}}J(M_1, M_j)\leq J(M_1,\alpha M_j) \leq  \alpha^{-\frac{2}{3(p-1)}}J(M_1,M_j), \label{control-KB2} \ee
In particular, the function $(M_1,M_j) \mapsto J(M_1,M_j)$ is continuous.
\end{lemma}
\begin{proof}
Let $M_1, M_j>0$ and $\alpha\in (0,1]$. Let $f\in\mathcal{F}(\alpha M_1, M_j)$ and let 
$$\tilde{f}(x,v)=\gamma f\left(\frac{\gamma^{1/3}}{\lambda^{1/3}}x,v\right)$$
be the unique rescaled function in $\mathcal{F}(M_1,M_j)$ given by Lemma \ref{lemresc}. Then we deduce from \fref{lambda-gamma} that $\lambda=\frac{1}{\alpha}$ and, since $\alpha<1$, we have also
$$\gamma^{q-1}\leq \alpha\leq \gamma^{p-1}.$$
Moreover, we also deduce from the rescaling identities of Appendix A that
$$K(\tilde{f})=\alpha^{1/3}\gamma^{-2/3}K(f).$$
This yields \fref{control-KB1}. The inequality \fref{control-KB2} is obtained similarly.
\end{proof}
\subsection{Proof of Theorem \ref{theo1}}\label{subsect2}

We are now ready to prove Theorem 1.1 which concerns the existence of the  minimizers and some of their properties. 

\bs
\ni
\textsl{Step 1. Existence of a minimizer.}
\nopagebreak

\bs
\ni
{\sl The Poisson-Manev case} ($\delta >0$).
\nopagebreak

\ms
\ni
Let $M_1,M_j>0$. From Lemma \ref{lemmonotone}, we know that $I(M_1,M_j)$ is finite. Consider a minimizing sequence $f_n$ of \eqref{def-I}:
\be \|f_n\|_{L^1}=M_1,\ \ \|j(f_n)\|_{L^1}=M_j\ \ \mbox{ and }\ \ \lim_{n\rightarrow+\infty}\mathcal{H}(f_n)=I(M_1,M_j). \label{hyp-fn}\ee
At fixed $n$, we denote by $f_n^{\ast x}$ the standard Schwarz rearrangement of $f_n$ with respect to the variable $x$. From the Riesz inequality, we have $E_{pot}(f_n^{\ast x}) \geq  E_{pot}(f_n)$. Thus we have $\mathcal{H}(f_n^{\ast x})\leq \mathcal{H}(f_n)$ and we may assume that the sequence $f_n$ is spherically symmetric in space.

We now observe that the sequence $(f_n)$ is bounded in $\mathcal{E}_j$. Indeed, from the subcritical condition \eqref{subcrit}, the kinetic energy of $f_n$ is controlled by the inequality \eqref{controlkinetic}. Thus, from Lemma \ref{propC2} of Appendix B, there exists $f\in \mathcal{E}_j$ such that 
$$f_n \rightharpoonup f \ \ \textrm{in}\ \ L^p(\R^6) \ \ \textrm{and} \ \ E_{pot}(f_n)\rightarrow E_{pot}(f).$$
By lower semi-continuity, we then have $\mathcal{H}(f)\leq I(M_1,M_j)$ which implies $f\neq 0$ since $I(M_1,M_j)<0$ (see \fref{encad-I}). Therefore there exist $0<\alpha,\beta\leq1$ such that $\|f\|_{L^1}=\alpha M_1$ and $\|j(f)\|_{L^1}=\beta M_j$. Combining this with \eqref{I-sup1} and \eqref{I-supj}, we get
$$\alpha^{\frac{4p-6}{3(p-1)}}\beta^\frac{1}{3(q-1)}I(M_1,M_j)\leq \calH(f)\leq I(M_1,M_j).$$
Hence $\alpha=\beta=1$ and $f$ is a minimizer of \eqref{def-I}.

\bs
\ni
{\sl The pure Manev case} ({$\delta =0$, $\kappa=1$}).
\nopagebreak

\ms
\ni
Let $f_n$ be a minimizing sequence of \eqref{def-K}. By a similar argument as above, we may assume that $f_n$ is spherically symmetric in space.
Moreover, from the rescaling formulas of Appendix \ref{appC}, the sequence of functions defined by $\hat{f}_n= f_n(\frac{x}{\lambda_n},\lambda_n v)$ (where $\lambda_n^2=\| |v|^2f_n\|_{L^1}$) satisfies 
$$\hat{f}_n \in  \mathcal{F}(M_1,M_j),\quad  K(\hat{f}_n)=K(f_n)\quad \mbox{and}\quad \| |v|^2\hat{f}_n\|_{L^1}=1.$$
In particular, $(\hat f_n)$ is bounded in $\mathcal{E}_j$. From Lemma \ref{propC2} of Appendix \ref{appC}, there exists $\hat f \in \mathcal{E}_j$ such that
$$\hat f_n\rightharpoonup \hat f \ \textrm{in}\  L^p(\R^6) \ \ \textrm{and}\ \ E_{pot}(\hat f_n)\rightarrow E_{pot}(\hat f).$$
Since $(\hat f_n)$ is a minimizing sequence of \eqref{def-K}, $E_{pot}(\hat f_n)$ converges to $J(M_1,M_j)^{-1}$. This implies that
$$E_{pot}(\hat f)=J(M_1,M_j)^{-1}>0$$
 and then $\hat f\neq 0$.

Moreover, from Fatou's Lemma, we have $K(\hat f) \leq \liminf K(\hat f_n)=J(M_1,M_j)$ and we have also $\hat f\in  \mathcal{F}(\alpha M_1, \beta M_j)$ with $0<\alpha,\beta \leq 1$. A similar rescaling as in the proof of Lemma \ref{lemmonotone-K} gives 
$$K(\hat f)\geq \frac{1}{\alpha^{\frac{q-3}{3(q-1)}}\beta^{\frac{2}{3(p-1)}}} J(M_1,M_j),$$
which implies that $\alpha=\beta=1$. Therefore $f$ is a minimizer of the variational problem \eqref{def-K}.

\bs
\ni
\textsl{Step 2. Euler-Lagrange equation for the minimizer.}
\nopagebreak

\bs
\ni
{\sl The Poisson-Manev case} ($\delta>0$).

\ms
\ni
Let $M_1,M_2>0$ satisfy the subcritical condition \eqref{subcrit} and let $Q$ be a minimizer of \eqref{def-I}.  Our goal in this step is to derive the Euler-Lagrange equation satisfied by $Q$. Let $\varepsilon>0$. We introduce the set
$$S_\varepsilon=\{ (x,v)\in \R^6,\ Q(x,v)\geq \varepsilon\},$$ 
and pick a compactly supported function $g\in L^{\infty}(\R^6)$  such that $g >0$ almost everywhere in $\R^6\backslash S_\varepsilon$. Then,
$$ \textrm{for all}\ t\in \left[0,\frac{\varepsilon}{\|g\|_\infty}\right],\qquad f_t=Q+t g \in \mathcal{E}_j\backslash \{0\}.$$
Similarly as in Appendix \ref{appA}, there exists a unique pair $(\gamma_t,\eta_t)$ positive numbers such that the function $\tilde{f_t}$ defined by 
$$\dis \tilde{f_t}(x,v)=\gamma_t f_t\left(x,\left(\frac{\gamma_t}{\eta_t}\right)^{\frac{1}{3}}v\right)$$ 
belongs to $\mathcal{F}(M_1,M_j)$, which is equivalent to
\be \eta_t=\frac{M_1}{\|f_t\|_{L^1}}\ \ \textrm{and}\ \ \ \frac{1}{\gamma_t}\left\| j(\gamma_t f_t)\right\|_{L^1}=\frac{M_j}{M_1}\|f_t\|_{L^1}. \label{nt-gammat} \ee
By differentiating the first equality, we obtain for $t\rightarrow 0$
\be \eta_t=1-\frac{\int_{\R^6} g}{M_1} t + o(t).\label{dev_eta}\ee
For all $t\in [0,\frac{\varepsilon}{\|g\|_\infty}]$ and for all $\gamma \in \R_+^\ast$, we set
$$G(\gamma,t)=\frac{1}{\gamma}\int_{\R^6}j(\gamma f_t)-\frac{M_j}{M_1}\int_{\R^6}f_t.$$
Then G is  clearly a $\mathcal{C}^1$ function of $t$ and  $\gamma$. Moreover, from Appendix \ref{appA}, we get $$\frac{\partial G}{\partial \gamma}(\gamma,t)>0.$$
This implies that $t\mapsto \gamma_t$ is a  $\mathcal{C}^1$ function and, by differentiating \eqref{nt-gammat} with respect to $t$, we obtain 
\be \gamma_t=1+\left(\frac{M_j}{M_1C_Q}\int_{\R^6} g-\frac{1}{C_Q}\int_{\R^6} j'(Q)g\right)t+o(t),\ \textrm{as}\ \  t\rightarrow 0,\label{dev_gamma} \ee
where, from the hypothesis \eqref{j-H3}, $C_Q=\|j'(Q)Q\|_{L^1}-M_j$ is a positive constant. 
Since $Q$ is a minimizer of \eqref{def-I} and since $\tilde{f_t}$ belongs to $\mathcal{F}(M_1,M_j)$, we have 
$$\lim_{t\rightarrow 0}\frac{\mathcal{H}(\tilde{f_t})-\mathcal{H}(Q)}{t}\geq 0.$$
From the computation in Appendix \ref{appA}, we also have
$$\left\{ \begin{array} {l}
\dis \left\| |v|^2\tilde{f_t}\right\|_{L^1}-\left\| |v|^2 Q\right\|_{L^1}=\frac{\eta_t^{5/3}}{\gamma_t^{2/3}}\left(\left\| |v|^2 Q\right\|_{L^1}+t \int_{\R^6}  |v|^2 g\right)-\left\| |v|^2 Q\right\|_{L^1}, \\
\dis E_{pot}(\tilde{f_t})-E_{pot}(Q)=\eta_t^2\left( E_{pot}(Q) -2t \int_{\R^6}\phi_Q g + t^2 E_{pot}(g)\right)-E_{pot}(Q).
\end{array}\right.$$
Inserting the expansions \eqref{dev_eta} and \eqref{dev_gamma} in these expressions, we get
\be \int_{\R^6} \left( \frac{|v|^2}{2}+\phi_Q -\lambda- \mu j'(Q) \right) g \geq 0, \label{lagrange-ineg}\ee
with 
$$\mu=-\frac{\left\| |v|^2 Q\right\|_{L^1}}{3 C_Q} \ \ \textrm{and}\  \lambda=-\frac{1}{M_1}\left( E_{pot}(Q)-\frac{\left\| |v|^2 Q\right\|_{L^1}}{6}\left(5+\frac{2M_j}{C_Q}\right)\right).$$
We now observe that $\mu$ and $\lambda$ are negative. Indeed, the equality $I(M_1,M_j)=\left\| |v|^2 Q\right\|_{L^1}-E_{pot}(Q)$ gives
$$\lambda=\frac{1}{M_1}\left( I(M_1,M_j)-\frac{\left\| |v|^2 Q\right\|_{L^1}}{6C_Q}\int_{\R^6} \left( j'(Q)Q-3j(Q)    \right)  \right),$$
where $I(M_1,M_j)<0$ and, from \eqref{j-H3},  $j'(Q)Q-3j(Q)\geq 0$. Since equality \eqref{lagrange-ineg} holds for all $g\in L^{\infty}(\R^6)$ which is compactly supported on $S_\varepsilon$, we have
$$ \frac{|v|^2}{2}+ \phi_Q - \lambda-  \mu j'(Q)=0\ \ \textrm{on}\  S_\varepsilon.$$
This means that this equality holds on Supp(Q). Similarly, out of the support of $Q$, as $g\geq 0$, we have
$$ \frac{|v|^2}{2}+ \phi_Q - \lambda \geq 0.$$
We finally get, for all $(x,v)\in \left(\R^3\right)^2$ ,$$Q(x,v)=(j')^{-1}\left(\dfrac{\frac{|v|^2}{2}+\phi_{Q}(x)-\lambda}{\mu}\right)_+.$$
We will show later (Step 3) that $\phi_Q$ is a $\mathcal C^1$ function, which is sufficient to ensure that $Q$ is a steady state. Indeed, from (H1), we deduce that the function $(j')^{-1}$ is $\mathcal C^1$ on $\RR_+^*$ with $(j')^{-1}(0)=0$. Hence, $Q$ being a function of the microscopic energy is a steady state of \fref{vm}, at least in the weak sense. Note that $Q$ is $\mathcal C^1$ in the interior of its support and is continuous but may have an infinite derivative at the boundary of its support.

\bs
\ni
{\sl The pure Manev case} ($\delta=0$, $\kappa=1$).

\ms
\ni
Let $Q$ be a minimizer of \eqref{def-K}. To get the Euler-Lagrange equation, we simply differentiate $f\mapsto  K(f)$ following the same procedure as above and find after computations 
$$Q=(j')^{-1}\left(\dfrac{\frac{|v|^2}{2}+\gamma \phi_{Q}(x)-\lambda}{\mu}\right)_+,$$
with $$\gamma=\frac{ \| |v|^2 Q \|_{L^1}}{E_{pot}(Q)}=J(M_1,M_j),\ \lambda=-\frac{(C_Q-2 M_j) \| |v|^2 Q\|_{L^1}}{6 M_1 C_Q}, \ \mu=-\frac{\| |v|^2 Q \|_{L^1}}{3C_Q}.$$
By inserting these expressions in \eqref{vm}, we observe that $f$ is a steady state of \eqref{vm} if and only if $\gamma=1$, which means $J(M_1,M_j)=1$. Let $M_1>0$ be fixed. From the control of the infimum \eqref{control-KB2}, we deduce that the function $M_j\mapsto J(M_1,M_j)$ is continuous, strictly decreasing and satisfies
$$\lim_{M_j\to 0}J(M_1,M_j)=+\infty, \qquad \lim_{M_j\to +\infty}J(M_1,M_j)=0.$$
Therefore, it is clear that there exists a unique $M_j$ such that $J(M_1,M_j)=1$.

\bs
\ni
\textsl{Step 3. Regularity of the potential $\phi_{Q}$ and compact support of $Q$}
\nopagebreak

\bs
\noindent
Let us prove that $\phi_{Q}$  belongs to $\mathcal{C}^{1,\alpha}$, for all $\alpha\in(0,1)$. Using the expression of $Q$, we get
\be \rho_{Q}(x)=\int_{\R^3}(j')^{-1}\left(\dfrac{\frac{|v|^2}{2}+\phi_{Q}(x)-\lambda}{\mu}\right)_+ dv. \label{expression-rho}\ee
Passing to the spherical velocity coordinate $u=|v|$ and performing the change of variable $w=\frac{u^2}{2|\mu|}$, we get
\be \rho_{Q}(x)=4\pi\sqrt{2}|\mu|^{\frac{3}{2}}\int_0^{+\infty}(j')^{-1}\left(k(x)-w\right)_+ \sqrt{w}dw,\ee
where $\dis k(x)=\frac{\phi_{Q}(x)-\lambda}{\mu}$. We remark that the support of $\rho_{Q}$ is contained in $\{x \in \RR^3,\ k(x)\geq 0\}$ and that $k(x)_+ \leq \phi_{Q}(x)/\mu$. Moreover, from (H2), for all $s\geq 0$ we have
$$\dis (j')^{-1}(s)\leq C \left(s^{\frac{1}{p-1}}+s^{\frac{1}{q-1}}\right).$$ 
Therefore
\bea
\rho_{Q}(x)&\leq& C \int_0^{k(x)_+}\left((k(x)-w)^{\frac{1}{p-1}}+(k(x)-w)^{\frac{1}{q-1}}\right)\sqrt{w}dw\nonumber\\
& \leq& C \left(\left(k(x)_+\right)^{\frac{3}{2}+\frac{1}{p-1}}+(\left(k(x)_+\right)^{\frac{3}{2}+\frac{1}{q-1}}\right)\nonumber\\
&\leq& C \left(\left|\phi_{Q}(x)\right|^{\frac{3}{2}+\frac{1}{p-1}}+\left|\phi_{Q}(x)\right|^{\frac{3}{2}+\frac{1}{q-1}}\right).\label{ineg-rho-phi}
\eea
Since $Q$ belongs $L^1\cap L^p$, $p>3$, and since $|v|^2Q$ belongs to $L^1$, we deduce from interpolation inequalities that $\rho_{Q}\in L^1 \cap L^{p_0}$ with $p_0=\frac{5p-3}{3p-1}\in(\frac{3}{2},\frac{5}{3}]$. 

Assume now that $\rho_{Q}\in L^1 \cap L^{k}$ for some $\frac{3}{2}<k<3$. Then from the Hardy-Littlewood-Sobolev inequality, we deduce that  $\phi_{Q}^{P}$ belongs to all $L^s$ with $3<s\leq \infty$ and that $\phi_{Q}^{M}$ belongs to all $L^{s}$ with $\frac{3}{2}<s\leq \frac{3k}{3-k}$. Hence, from \eqref{ineg-rho-phi}, since $q\geq p$, we deduce that the function
$\rho_{Q}$ belongs to $L^{\ell}$ with
\be \ell\left(\frac{3}{2}+\frac{1}{p-1}\right)=\frac{3k}{3-k}.\label{suite}\ee
Using \fref{suite} and $p>3$ (Assumption (H2)), a simple bootstrap argument enables to prove that there exists $r>3$ such that  $\rho_Q\in L^r$.

Consequently, from Sobolev embeddings and from $(-\Delta)^{1/2}\phi_{Q}^{M}=-\rho_Q$, we deduce that  the Manev potential $\phi_{Q}^{M}$ belongs to $\mathcal{C}^{0,\alpha}$ for all $\alpha\in(0,1-\frac{3}{r})$. Since this function converges to 0 at the infinity, we have  $\phi_Q^M\in L^\infty$, and then $\phi_Q\in L^\infty$. Thus \eqref{ineg-rho-phi} gives $\rho_Q\in L^\infty$. Finally, using again Sobolev embeddings, $\phi_{Q}^{M}$ and $\phi_{Q}^{P}$ belong to $\mathcal{C}^{0,\alpha}$ for all $\alpha\in(0,1)$. 

From the regularity of $\phi_Q$, the fact that this function goes to 0 as $|x|\to +\infty$ and that $\lambda<0$, one deduces that $$\mbox{Supp}(Q)=\left\{(x,v)\in \R^6,\quad \frac{ |v|^2}{2}+\phi_{Q}(x)-\lambda\leq 0\right\}$$ is a compact subset of $\R^6$.

Let us now prove that $\rho_Q$ belongs to $\mathcal{C}^{0,\alpha}$ for all $\alpha\in(0,1)$. Passing to the spherical coordinate in velocity in the expression \fref{expression-rho} of $\rho_Q$ and performing the change of variable $s=\frac{\frac{|v|^2}{2}+\phi_{Q}(x)-\lambda}{\mu}$ yields
$$\rho_Q(x)= |\mu|^\frac{3}{2}\sqrt{2} \int_{\RR_+} (j')^{-1}(s) \left( \frac{\phi_Q(x)-\lambda}{\mu} -s \right)^{1/2}_+ ds. $$
For all $k \in \R$,  denote $f(k)= \int_{\RR_+} (j')^{-1}(s) \left( k-s \right)^{1/2}_+ ds$. We claim that, for all $k_0>0$, we have
\be \forall k_1,k_2 \in (-\infty, k_0], \qquad  \left| f(k_1)-f(k_2)\right| \leq (j')^{-1}(k_0) k_0^{1/2}\left| k_1-k_2 \right|. \label{claim-reg}\ee
By taking $k_0=\frac{-\| \phi_Q \|_\infty -\lambda}{\mu}$, we deduce from this claim that
$$\forall (x,y)\in \RR^6,\qquad \left|\rho_{Q}(x)-\rho_Q(y)\right|\leq C\left|\phi_Q(x)-\phi_Q(y)\right|.$$ 
This shows that $\rho_Q\in \mathcal{C}^{0,\alpha}$ for all $\alpha\in(0,1)$. Next, since we have $$\phi_Q=\Delta^{-1}(\rho_Q)-(-\Delta)^{-1/2}(\rho_Q),$$
we can conclude from standard regularity argument that $\phi_Q\in \mathcal{C}^{1,\alpha}$ for all $\alpha\in(0,1)$. This is  the regularity of the potential stated in Theorem \ref{theo1}. Let us now prove the claim \eqref{claim-reg}.
For all $k_1\leq k_2\leq k_0$, we have
$$\begin{array} {rcl}
 f(k_2)-f(k_1)  & =  & \dis \int_{\RR_+} (j')^{-1}(s) \left(( k_2-s )^{1/2}_+ - ( k_1-s )^{1/2}_+ \right)ds \\
   & \leq & \dis (j')^{-1}(k_0)  \int_{\RR_+}  \left(( k_2-s )^{1/2}_+ - ( k_1-s )^{1/2}_+ \right)ds \\
   &  &\dis = \frac{2}{3}  (j')^{-1}(k_0) \left( (k_2)_+^{3/2}-(k_1)_+^{3/2} \right)\leq (j')^{-1}(k_0)k_0^{1/2}(k_2-k_1).
\end{array} $$
Since $f$ is an increasing function, this yields \fref{claim-reg}. This concludes the proof of the regularity of the potential stated in Theorem \ref{theo1}.

\bs
\ni
\textsl{Step 4. The functions $\rho_Q$ and $\phi_{Q}$ are spherically symmetric and monotone.}
\nopagebreak

\bs
\noindent
Consider a minimizer $Q$ of \eqref{def-I}, continuous and compactly supported thanks to the previous step, and denote by $Q^{\ast x}$  its symmetric rearrangement with respect to the $x$ variable only. We have clearly $Q^{\ast x} \in \mathcal{F}(M_1,M_j)$ and $\int |v|^2 Q dxdv=\int |v|^2 Q^{\ast x} dxdv$. Moreover,  by the Riesz inequality (see \cite{LL}), we have
\be
\label{riesz}
\int_{\RR^{6}} Q(x,v)Q(y,w)g(|x-y|)dxdy\leq \int_{\RR^{6}} Q^{\ast x}(x,v)Q^{\ast x}(y,w)g(|x-y|)dxdy 
\ee
for all $(v,w)\in \RR^3\times \RR^3$, where $g(r)=\frac{\delta}{r}+\frac{\kappa}{r^2}$ (recall that $\delta\geq 0$ and $\kappa\geq 0$). Therefore, by integrating this inequality with respect to $v$ and $w$, one gets
$$E_{pot}(Q)\leq E_{pot}(Q^{\ast x}),$$
which means that $\calH(Q^{\ast x})\leq \calH(Q)$: $Q^{\ast x}$ is also a minimizer of \eqref{def-I}. Hence, we must have equality in the above inequalities: $E_{pot}(Q)= E_{pot}(Q^{\ast x})$ and, even more, we have an equality in \eqref{riesz} for all $v,w$. We are then in a situation of equality in the Riesz inequality: since the function $g$ is strictly decreasing, we deduce that (see \cite{LL}), for all $v,w$, there exists a translation shift $x_0(v,w)$ such that
\be
\label{shift}
Q(x,v)=Q^{\ast x}(x+x_0(v,w),v)\quad\mbox{and}\quad  Q(x,w)=Q^{\ast x}(x+x_0(v,w),w).
\ee
Let $v$ be such that $Q(\cdot,v)\not \equiv 0$.  $Q$ being compactly supported, we  integrate the first equality in \eqref{shift} against $x$ and obtain
\bee
\int_{\RR^3}xQ(x,v)dx&=&\int_{\RR^3}xQ^{\ast x}(x+x_0(v,w),v)dx\\
&=&\int_{\RR^3}xQ^{\ast x}(x,v)dx-x_0(v,w)\int_{\RR^3}Q^{\ast x}(x,v)dx.
\eee
Hence, we have the expression
$$x_0(v,w)=\frac{\int_{\RR^3}x(Q(x,v)-Q^{\ast x}(x,v))dx}{\int_{\RR^3}Q^{\ast x}(x,v)dx}$$
and then $x_0(v,w)$ is independent of $w$. Similarly, using the second equality in \eqref{shift}, one obtain that $x_0$ is independent of $v$. We have proved finally that there exists $x_0\in \RR^3$ such that
$$Q(x,v)=Q^{\ast x}(x+x_0,v)=Q^{\ast x}(|x+x_0|,v),\qquad \forall v\in \RR^3.$$
Consequently, up to a translation shift, $\rho_Q$ is a nonincreasing function of $|x|$.

\bs
Let us now prove that $\phi_{Q}$ is a nondecreasing function  of $r=|x|$. Since the function $j$ is convex and $\mu<0$, the expression \eqref{expression-rho} shows that $\phi_{Q}(r)$ is nondecreasing on the compact support  of the nonincreasing function $\rho_Q(r)$. Let $[0,R_Q]$ be this compact support.

For $|x|=r>R_{Q}$, we have
$$\phi_{Q}^P(x) =- \int_{\R^3}\frac{\rho_{Q}(y)}{4 \pi |x-y|}dy \ \ \textrm{and} \  \ \phi_{Q}^M(x) =- \int_{\R^3}\frac{\rho_{Q}(y)}{2 \pi^2|x-y|^2}dy.$$
Passing to the spherical coordinate (see the proof of Proposition \ref{propC1} in Appendix \ref{appC}), we have
$$ \phi_{Q}^P(x) =- \frac{M_1}{4\pi r}\quad \mbox{and}\quad \phi_{Q}^M(x) =- \frac{1}{\pi}  \int_{0}^{R_{Q}}\frac{s\rho_{Q}(s)}{r} \ln \left(\frac{r+s}{r-s}\right)ds.$$
Since the function $r\mapsto \frac{1}{r} \ln \left(1+\frac{2s}{r-s}\right)$ is positive and decreasing, $\phi_{Q}$ is nondecreasing on $[R_Q,+\infty)$. The proof of Theorem \ref{theo1} is complete.
\qed
 

\section{Orbital stability of the ground states}
To prove the orbital stability result  stated in Theorem \ref{theo3}, we first need to prove the uniqueness of the minimizer under equimeasurability and symmetric constraints which are inherited from the invariance properties of the Vlasov-Manev flow. This uniqueness result is at the heart of our stability analysis and is quite robust in the sense that its proof does not use the Euler-Lagrange equation.  Technically, the uniqueness proof only uses the fact that a minimizer is a function of a certain microscopic energy, which is not necessarily that of the minimizer. Therefore our proof does not use the equation satisfied by the potential itself (a non linear fractional-Laplacian equation in the present case).
\label{sect2}
\subsection{Uniqueness of the minimizer under equimeasurability condition}
\label{subsect3}
This subsection is devoted to the proof of Lemma \pref{theo2}.

\bs
\ni
Let $$Q_0(x,v)=F\left(\frac{|v|^2}{2}+\psi_0(x)\right),\qquad Q_1(x,v)=F\left(\frac{|v|^2}{2}+\psi_1(x)\right)$$
 be the functions defined in Lemma \ref{theo2}. Note that $\psi_0$ and $\psi_1$ are not supposed to coincide with $\phi_{Q_0}$ and $\phi_{Q_1}$ respectively, which means that they are not supposed to satisfy the fractional-Laplacian equation. For $i\in\{0,1\}$ and for all $\tau<0$, we define
 $$a_{\psi_i}(\tau)=\mbox{meas}\left\{(x,v)\in \RR^6,\quad \frac{|v|^2}{2}+\psi_i(x)<\tau\right\}.$$ From the equimeasurability of $Q_0$ and $Q_1$ and the properties of the function $F$, we have
\be \forall \tau<0, \quad a_{\psi_0}(\tau)=a_{\psi_1}(\tau). \label{egal-aphi} \ee
 For $i\in\{0,1\}$, we define  $$\mu_{\psi_i}(\lambda)=\mbox{meas}\left\{x\in \RR^3,\quad \psi_i(x)<\lambda\right\}$$ for all $\lambda< 0$ and we have then for all $\tau<0$,
 $$ a_{\psi_i}(\tau)=\int_{\R^3}\mu_{\psi_i}\left(\tau-\frac{|v|^2}{2}\right)dv. $$
Passing to the spherical velocity coordinate $u=|v|$ and performing the change of variable $w=\tau-u^2/2$, we obtain 
\be a_{\psi_i}(\tau)=4\pi\sqrt{2}\int^{\tau}_{-\infty}\mu_{\psi_i}(w)\sqrt{\tau-w}\,dw. \label{expression-aphi}\ee
We claim that the expression \eqref{expression-aphi} and the equality \eqref{egal-aphi} imply that, 
\be \textrm{for almost all }  \lambda<0,\qquad \mu_{\psi_0}(\lambda)=\mu_{\psi_1}(\lambda). \label{egal-muphi}\ee 
Hence, as $\psi_0$ and $ \psi_1$ are continuous and nondecreasing, we have $\psi_0=\psi_1$ on the set
$$\{x\in \R^3,\ \psi_0(x)<0\}=\{x\in \R^3,\ \psi_1(x)<0\},$$
which immediatly gives $Q_0=Q_1$.

\bs
\ni
{\em Proof of \eqref{egal-muphi} from  \eqref{egal-aphi} and \eqref{expression-aphi}.}  By differentiating with respect to $\tau$ the function $a_{\psi_i}$ defined by \eqref{expression-aphi}, one gets 
\be \forall \tau<0,\qquad  a_{\psi_i}'(\tau)=2\pi\sqrt{2}\int^{\tau}_{-\infty}\frac{\mu_{\psi_i}(w)}{\sqrt{\tau-w}}dw. \ee
Now, remarking that, for $w<\lambda$, the following integral is constant:
$$I(\lambda,w)=\int^\lambda_w \frac{d\tau}{\sqrt{(\lambda-\tau)(\tau-w)}}=\pi,$$
one deduces from the Fubini theorem that
$$\int^{\lambda}_{-\infty} \frac{a_{\psi_i}'(\tau)}{\sqrt{\lambda-\tau}}d\tau=2\pi\sqrt{2}\int^{\lambda}_{-\infty}\mu_{\psi_i}(w)I(\lambda,w)dw=2\pi^2\sqrt{2}\int^{\lambda}_{-\infty}\mu_{\psi_i}(w)dw.$$
Thus, from $a_{\psi_0}=a_{\psi_1}$, we deduce that $\mu_{\psi_0}(\lambda)=\mu_{\psi_1}(\lambda)$ for almost all $\lambda<0$, and the proof of \eqref{egal-muphi} is complete.

\bs
\ni
{\em End of the proof of Lemma \pref{theo2}.} Let $Q_0, Q_1$ be two equimeasurable and spherically symmetric steady states to \eqref{vm} which minimize the variational problem \eqref{def-I} in the Poisson-Manev case ($\delta > 0$) or the variational problem \eqref{def-K} in the pure Manev case ($\delta=0$, $\kappa=1$). From Theorem \ref{theo1}, there exist $\lambda_0,\mu_0,\lambda_1,\mu_1<0$ such that, for $i\in\{0,1\}$,
\be Q_i(x,v)=(j')^{-1}\left(\dfrac{\frac{|v|^2}{2}+\phi_{Q_i}(x)-\lambda_i}{\mu_i}\right)_+.\ee
We now define, for $i\in\{0,1\}$, $$\widetilde Q_i(x,v)=Q_i\left(\frac{x}{|\mu_i|^{1/2}},|\mu_i|^{1/2} v\right).$$
The function $\widetilde Q_0$ and $\widetilde Q_1$ are still equimeasurable and satisfy
$$\widetilde Q_i(x,v)=(j')^{-1}\left( -\frac{|v|^2}{2}-\psi_i(x)\right)_+ \ \ \textrm{with}\ \ \psi_i(x)=\frac{\phi_{Q_i}\left(\frac{x}{|\mu_i|^{1/2}}\right)-\lambda_i}{|\mu_i|}.$$
Since $\phi_{Q_i}$ is continuous nondecreasing and converges to $0$ as $r\rightarrow +\infty$,  the function $\psi_i$ is continuous, nondecreasing and the set $\{x\in \R^3,\ \psi_0(x)<0\}$ is bounded. From the previous step, we then conclude that  
$$\widetilde Q_0=\widetilde Q_1,$$ 
which means that
\be \label{eqE2}
Q_1\left(x,v\right)=Q_0\left(\frac{x}{\alpha},\alpha v\right) \ \ \textrm{with}\ \ \alpha=\sqrt{\frac{\mu_1}{\mu_0}}. \ee
We shall now prove that $\alpha=1$.

\ms
\ni
{\sl The pure Manev case}. In this case, the equality of the kinetic energies (which is assumed in this lemma) directly gives $\alpha=1$.

\ms
\ni
{\sl The Poisson-Manev case}. Let us derive a virial identity satisfied by the minimizers $Q$ of \eqref{def-I}, using a rescaling argument. For $\lambda>0$, we set $f_\lambda (x,v)=Q(\lambda x,\frac{v}{\lambda})$, which implies $f_\lambda\in \mathcal{F}(M_1,M_j)$ and 
$$\mathcal{H}(f_\lambda)=\lambda^2 \left\||v|^2 Q\right\|_{L^1} -\lambda \delta E_{pot}^P(Q)-\lambda^2 \kappa E_{pot}^M(Q).$$
This function of $\lambda$ has a strict global minimizer in $\lambda=1$, which yields the following virial identity:
\be \left\||v|^2 Q\right\|_{L^1}- \frac{\delta}{2}E_{pot}^P(Q)-\kappa E_{pot}^M(Q)=0. \ee
Moreover we recall that $Q$ satisfies
$$ \left\||v|^2 Q\right\|_{L^1}- \delta E_{pot}^P(Q)-\kappa E_{pot}^M(Q)=I(M_1,M_j).$$
Combining this two equalities,  we get 
\be
\label{eqE}
\frac{\delta}{2}E_{pot}^P(Q)=-I(M_1,M_j).
\ee
Let us now use this identity for the two minimizers $Q_0$ and $Q_1$. From \eqref{eqE} and \eqref{encad-I}, one deduces that
$$E_{pot}^P(Q_0)=E_{pot}^P(Q_1)>0.$$
Moreover, from \eqref{eqE2} and Appendix \ref{appA}, one gets
$$E_{pot}^P(Q_1)=\frac{1}{\alpha} E_{pot}^P(Q_0).$$ This yields $\alpha=1$, which ends the proof of Lemma \ref{theo2}.
\qed

\subsection{Orbital stability of the minimizers, proof of Theorem \pref{theo3}}
\label{subsect4}
In this subsection, we prove Theorem \pref{theo3}.

\ms
\ni
{\sl The Poisson-Manev case}. 

\ms
\ni
Let $Q$ be a minimizer of \eqref{def-I} and assume that Theorem \ref{theo3} is false. Then there exist $\varepsilon>0$ and sequences $f_0^n \in \mathcal{E}_j$, $t_n>0$, such that
\be \lim_{n\rightarrow + \infty}\|f_0^n-Q\|_{\mathcal{E}_j}=0, \label{conv-f0}\ee
and
\be \forall n\geq0,\ \forall x_0 \in \R^3,\ \|f^n(t_n,x,v)-Q(x+x_0,v)\|_{\mathcal{E}_j}\geq \varepsilon, \label{nonconv-fn}\ee
where $f^n(t,x,v)$ is a solution to \eqref{vm} with initial data $f_0^n$.
 
From \eqref{conv-f0}, we have
\be \lim_{n\rightarrow+\infty}\mathcal{H}(f_0^n)=I(M_1,M_j),\ \lim_{n\rightarrow+\infty}\|f_0^n\|_{L^1}=M_1,\ \lim_{n\rightarrow+\infty}\|j(f_0^n-Q)\|_{L^1}=0. \ee
In particular, $f_0^n$ converges to $Q$ in the strong $L^p$ topology and hence almost everywhere, up to a subsequence. Using the assumptions (H1), (H2) and the convexity of $j$, we deduce from a classical argument (see Theorem 2 in \cite{BL}) that $\|j(f_0^n)\|_{L^1}\rightarrow \|j(Q)\|_{L^1}$.

Let now $g_n(x,v)=f_n(t_n,x,v)$. By the conservation properties of the Vlasov-Manev system \eqref{vm},  we have
\be 
\lim_{n\rightarrow+\infty}\mathcal{H}(g_n)= I(M_1,M_j),\ \lim_{n\rightarrow+\infty}\|g_n\|_{L^1}=M_1,\ \lim_{n\rightarrow+\infty}\|j(g_n)\|_{L^1}=M_j, \label{conservations}
\ee
and, for all $t\geq 0$, 
\be
\mbox{meas}\{(x,v)\in\R^6,\ g_n(x,v)>t\} =  \mbox{meas}\{(x,v)\in\R^6,\ f_0^n(x,v)>t\}. \label{equi-gn-f0n} 
\ee
{}From Appendix \ref{appA}, let us define $$\bar{g}_n(x,v)=\gamma_n g_n\left( \frac{\gamma_n^{1/3}}{\lambda_n^{1/3}}x, v\right)$$ such that $\|\bar{g}_n\|_{L^1}=M_1$ and $\|j(\bar{g}_n)\|_{L^1}=M_j$. Then, from \eqref{conservations},
\be
\label{coefcv}
\gamma_n \rightarrow 1, \ \lambda_n \rightarrow 1,
\ee
and $$\lim_{n\rightarrow+\infty}\mathcal{H}(\bar{g}_n)=\lim_{n\rightarrow+\infty}\mathcal{H}(g_n)= I(M_1,M_j).$$
Hence $\bar{g}_n$ is a minimizing sequence of \eqref{def-I}. Now, from classical arguments based on concentration-compactness techniques (\cite{L1}, \cite{L2}) and the non-dichotomy inequality \eqref{nondicho-I} (see \cite{LMR} and \cite{LM2} for more details), $\bar{g}_n$ is relatively strongly compact in $\mathcal{E}_j$ and converges to a ground state $Q_1$, up to a subsequence and up to a translation shift in space. Hence, by \eqref{coefcv}, we have
\be g_n \rightarrow Q_1 \mbox{ in } \mathcal{E}_j \label{conv-gnbar}\ee
up to a subsequence and up to a translation shift.
 
Let us now prove that the equimeasurability \eqref{equi-gn-f0n} and the $L^1$ convergences of $g_n$ and $f_0^n$ imply the equimeasurability of $Q$ and $Q_1$. Indeed, we remark that, for $t>0$ and $0<\varepsilon<t$,  
$$\left\{\begin{array} {l}
 \dis  \{ g_n>t  \}   \subset  \left( \{ |g_n-Q_1|<\varepsilon \} \cap \{ Q_1>t-\varepsilon\}\right) \cup  \{ |g_n-Q_1|\geq \varepsilon \}, \\ \\
  \dis  \{ g_n>t  \}   \supset  \{ |g_n-Q_1|<\varepsilon \} \cap \{ Q_1>t+\varepsilon\}.
 \end{array}\right.$$
By passing to the limit as $n\rightarrow+\infty$, one gets
$$\left\{\begin{array} {l}
 \dis \limsup_{n\rightarrow+\infty}  \mbox{meas} \{ g_n>t  \}  \leq \mbox{meas}\{ Q_1>t-\varepsilon\}, \\ \\
  \dis \liminf_{n\rightarrow+\infty}  \mbox{meas} \{ g_n>t  \}  \geq \mbox{meas}\{ Q_1>t+\varepsilon\}.
 \end{array}\right.$$
Finally, passing to the limit as $\varepsilon \rightarrow 0$, we have $ \mbox{meas}\{ g_n>t  \}\rightarrow \mbox{meas}\{ Q_1>t\}$ for almost all $t>0$ and similarly $\mbox{meas}\{ f_n^0>t  \}\rightarrow \mbox{meas}\{ Q>t\}$ for almost all $t>0$. Observing that the functions $t\mapsto \mbox{meas}\{ Q>t\}$ and $t\mapsto \mbox{meas}\{ Q_1>t\}$ are right-continuous,  we obtain the equimeasurability of $Q$ and $Q_1$.

We now use the characterization of ground states stated in Theorem \ref{theo1} and the uniqueness result given by Lemma \ref{theo2}, to conclude that, $Q=Q_1$, up to a space translation shift. Finally, \eqref{conv-gnbar} contradicts \eqref{nonconv-fn} and the proof of Theorem \ref{theo3} is complete.
\qed

\bs
\ni
{\sl The pure Manev case}.

\ms
\ni
To prove Theorem \ref{theo3} for the pure Manev case, it is clearly sufficient to prove the following proposition.
\begin{proposition}\label{propstab} Let $Q$ be a steady state of \eqref{vm} which minimizes \eqref{def-K} and let $(f_n)_{n\geq 1}\in \mathcal{E}_j$ such that\\
\be
\label{hyp1}
\forall s>0, \quad \lim_{n\to +\infty}\mbox{\rm meas}\{(x,v)\in\R^6,\ f_n(x,v)>s\} =  \mbox{\rm meas}\{(x,v)\in\R^6,\ Q(x,v)>s\}
\ee
and
\be
\label{hyp2}
\|f_n\|_{L^1} \rightarrow  \| Q\|_{L^1},\quad  \limsup_{n\rightarrow +\infty} \|j(f_n)\|_{L^1} \leq \| j(Q) \|_{L^1}\mbox{ and }\limsup \frac{\mathcal{H}(f_n)}{\| |v|^2 f_n \|_{L^1}} \leq 0.
\ee
Then there exists $(y_n)_{n\geq 1}$ sequence on $\R^3$ such that up to a subsequence
$$ f_n\left(\lambda_n(x+y_n),\frac{v}{\lambda_n}\right) \rightarrow Q \ \textrm{in}\  \mathcal{E}_j,\ \  \textrm{where} \ \ \lambda_n=\left(\frac{ \| |v|^2 Q \|_{L^1}}{ \| |v|^2 f_n \|_{L^1}}\right)^{1/2}.$$
\end{proposition}
\begin{proof}
From the assumption (ii), the sequence of rescaled functions defined by $\hat{f_n}(x,v)=f_n(\lambda_n x, \frac{v}{\lambda_n})$ satisfies
\be
\label{hyp3}
	\forall s>0, \quad \lim_{n\to +\infty}\mbox{\rm meas}\{(x,v)\in\R^6,\ \hat f_n(x,v)>s\} =  \mbox{\rm meas}\{(x,v)\in\R^6,\ Q(x,v)>s\},
	\ee
	\be
	\label{hyp4}
	\limsup \mathcal{H}(\hat{f_n}) \leq 0 \ \textrm{and then}  \ \liminf E_{pot}(\hat{f_n}) \geq \| |v|^2 Q \|_{L^1}= E_{pot}(Q)>0. \ee
From concentration-compactness argument \cite{L1,L2} and using Lemma 3.2 in \cite{LMR}, one can deduce that the sequence $\hat f_n$ satisfies one of the three following alternatives: compactness, vanishing or dichotomy, see e.g. Lemma 3.2 in \cite{LMR} for the definitions of these standard notions. In fact, we shall prove that only compactness may occur.

Indeed, vanishing cannot occur, since \fref{hyp4} prevents $E_{pot}(\hat{f_n}) $ from going to $0$ as $n\to +\infty$. Next, if dichotomy occurs (see \cite{LMR}), then there exist $0<\alpha<1$ such that, for all $\eps>0$, there exists a decomposition $\hat{f_n}=f_n^1+f_n^2+ w_n$, with disjoint supports,  such that we have
\be \left| \|f_n^1\|_{L^1}-\alpha \|Q\|_{L^1}\right|+\left| \|f_n^2\|_{L^1}-(1-\alpha)\|Q\|_{L^1}\right|<\varepsilon, \label{dic}\ee
and
\be \left| E_{pot}(\hat{f_n})-E_{pot}(f_n^1)-E_{pot}(f_n^2) \right|< \varepsilon.
\label{erreurepot}\ee
The control of the mass \eqref{dic} and the monotonicity of the infimum from Lemma \ref{lemmonotone-K} imply that
$$K(f_n^1)\geq J( \alpha\|Q\|_{L^1}+\eps,\|j(f_n^1)\|_{L^1})\geq J( \alpha\|Q\|_{L^1}+\eps, \|j(Q)\|_{L^1}).$$
By choosing $\eps<\frac{1-\alpha}{2}\|Q\|_{L^1}$, we ensure that
$$J( \alpha\|Q\|_{L^1}+\eps, \|j(Q)\|_{L^1})\geq J\left( \frac{1+\alpha}{2}\|Q\|_{L^1}, \|j(Q)\|_{L^1}\right)>1$$
and then
\be \mathcal{H}(f_n^1)=\left\| |v|^2 f_n^1 \right\|_{L^1}\left( 1-\frac{1}{K(f_n^1)}  \right)\geq C_1\left\| |v|^2 f_n^1 \right\|_{L^1}, \label{astuce}\ee
where $C_1>0$ does not depend of $\varepsilon$ and $n$, which gives 
\be
\label{eq} \liminf_{n\rightarrow +\infty}  \mathcal{H}(f_n^1)\geq 0, \quad \mbox{ and similarly}\quad \liminf_{n\rightarrow +\infty}  \mathcal{H}(f_n^2)\geq 0.
\ee
Moreover, we have
\bee
\calH(\hat f_n) &=&\left\| |v|^2 f_n^1 \right\|_{L^1}+\left\| |v|^2 f_n^2 \right\|_{L^1}+\left\| |v|^2 w_n \right\|_{L^1}-E_{pot}(f_n)\\
&\geq&\calH(f_n^1)+\calH(f_n^2)-\eps,
\eee
where we used \eqref{erreurepot}. Passing to the limit in this inequality as $n\to +\infty$, we obtain
$$
\limsup_{n\to +\infty}(\calH(f_n^1)+\calH(f_n^2))\leq \eps.
$$
From \fref{astuce}, we deduce that
$$
\limsup_{n\to +\infty}(\left\| |v|^2 f_n^1 \right\|_{L^1}+\left\| |v|^2 f_n^2 \right\|_{L^1})\leq C\eps,
$$
where $C>0$ is independent of $\eps$. Then, using
$$E_{pot}(f_n^1)+E_{pot}(f_n^2)=\left\| |v|^2 f_n^1 \right\|_{L^1}+\left\| |v|^2 f_n^2 \right\|_{L^1}-\calH(f_n^1)-\calH(f_n^2)$$
together with \fref{eq}, we get
$$
\limsup_{n\to +\infty}(E_{pot}(f_n^1)+E_{pot}(f_n^2))\leq C\eps.
$$
For $\eps$ small enough, this contradicts \fref{hyp4} and \fref{erreurepot}. This proves that dichotomy cannot occur and then compactness follows. In particular, there exists a sequence of translation shifts $y_n$ such that, up to a subsequence,
$$\hat f_n(\cdot+y_n)\rightharpoonup \hat f\mbox{ in }L^1(\RR^6)\quad \mbox{and}\quad E_{pot}(\hat f_n)\to E_{pot}(\hat f).$$
Moreover, by lower semicontinuity and by \fref{hyp4}, we have
$$
\|j(\hat f)\|_{L^1}\leq \|j(Q)\|_{L^1}\quad \mbox{and}\quad K(\hat f)\leq \limsup_{n\to +\infty} K(\hat f_n)\leq 1.
$$
Therefore, by \fref{control-KB2}, we have
$$1=J\left(\|Q\|_{L^1},\|j(Q)\|_{L^1}\right)\leq J\left(\|Q\|_{L^1},\|j(\hat f)\|_{L^1}\right)\leq K(\hat f)\leq 1.$$
The strict monotonicity of the function $M_j\mapsto J(M_1,M_j)$ yields
$$\|j(\hat f)\|_{L^1}= \|j(Q)\|_{L^1} \quad \mbox{and}\quad K(\hat f)=1.$$
From this, it is now standard to conclude the strong convergence $\hat f_n(\cdot+y_n)\to \hat f$ in $\calE_j$. Note that $\hat f$ is a minimizer of \fref{def-K} satisfying $\left\| |v|^2 \hat f \right\|_{L^1}=\left\| |v|^2 Q \right\|_{L^1}$.
Furthermore, from the strong $L^1$ convergence of $\hat f_n(\cdot+y_n)$ to $Q$ and from their equimeasurability deduced from \fref{hyp1}, one can prove that $\hat f$ is equimeasurable to $Q$ (this proof can be done following the same lines as in the above proof of orbital stability for the Poisson-Manev case). Therefore, from Lemma \ref{theo2}, one deduces finally that $\hat f$ is equal to $Q$, up to a translation shift. This concludes the proof of Proposition \ref{propstab} and the proof of Theorem \ref{theo3} is complete.
\end{proof}

\section{Self-similar solutions in the pure Manev case}\label{sectselfsim}

From now on, we only consider the pure Manev case ($\delta=0$, $\kappa=1$). This section is devoted to the proof of Theorem \ref{thm6}. Let $Q$ be a steady state solution to \fref{vm} which minimizes \fref{def-K}.

We seek, for $b$ small enough, a compactly supported  and spherically symmetric stationary profile $Q_b\in \mathcal{C}^0(\R^6)$, with $\phi_{Q_b}\in \mathcal C^1(\RR^3)$, such that
\be Q_b\left(\frac{x}{\lambda(t)},\lambda(t)v\right) \ \ \textrm{with} \ \lambda(t)=\sqrt{2b(T-t)} \label{selfsim}\ee 
is a solution to \eqref{vm} in the pure Manev case. We insert \fref{selfsim} in \fref{vm}, use the identity $\dot{\lambda}=-b/\lambda$ and then get that $Q_b$ has to satisfy  (at least in the weak sense) the following equation in the self-similar variables $(\tilde{x},\tilde{v})=(\frac{x}{\lambda(t)},\lambda(t)v)$ (which are renoted $(x,v)$ for simplicity),
\be v\cdot \nabla_x Q_b - \nabla_x \phi_{Q_b} \cdot \nabla_v Q_b +b \left(x\cdot \nabla_x Q_b -v \cdot \nabla_v Q_b\right)=0.\label{jeans}\ee
We first observe that a function of the form 
$$g(x,v)=F\left(  \frac{|v|^2}{2}+b x\cdot v + \phi_g(x)\right)$$
satisfies this equation. However, for non trivial profiles $F$ and for $b\neq 0$, it can be seen that such a function does not belong to $\mathcal{E}_j$ (it has always infinite mass and energy). To solve this problem we proceed as in \cite{merle-raphael,jams} and introduce a radial cut-off function $\chi$ from $\R^3$ to $[0,1]$ such that 
$$	\chi(x)=1 \ \textrm{for}\  |x|<r_\chi \ \ \textrm{and}\ \    \chi(x)=0 \ \textrm{for}\  |x|>R_\chi=2r_\chi,$$
where $r_\chi>0$ will be defined later on (see \fref{choix rchi}). We shall prove the existence of a function having the form
\be
\label{desired}Q_b(x,v)=F\left(  \frac{|v|^2}{2}+b\chi(x) x\cdot v + \phi_{Q_b}(x)\right),
\ee
which is compactly supported in $\{|x|<r_\chi\}$. Here, the function $\phi_{Q_b}$ belongs to $\mathcal C^1$ and the function $F$ is a continuous nonnegative function on $\RR$, which is $\mathcal C^1$ on $]-\infty,e_0[$ for some $e_0<0$ and vanishes on $[e_0,+\infty[$. Hence, we have
$$Q_b(x,v)=F\left(  \frac{|v|^2}{2}+bx\cdot v + \phi_{Q_b}(x)\right),\qquad \forall (x,v)\in \RR^6\mbox{ such that }|x|<r_\chi,$$
which is sufficient to deduce that $Q_b$ is a solution to \eqref{jeans}. 

To construct such self-similar profile $Q_b$, it is natural to use a minimization problem with constraints. However, if the number of constraints is finite, which is the case for instance if we prescribe the mass and a Casimir functional as in Section \ref{sect1}, then the uniqueness of the minimizer is not garanted. This uniqueness property will be crucial to ensure that  $Q_b$ is in the vicinity of $Q$. Therefore, we will choose a variational problem with an infinite number of constraints which, using Lemma \ref{theo2}, will lead to a unique minimizer. More precisely, we define the following set of constraints:
\be
\label{eq-def}
\mbox{Eq}(Q)=\left\{f\in \calE_j:\,f \mbox{ is equimeasurable with }Q\right\}.
\ee
Then we consider the associated variational problem
\be T_b :=\inf \left\{T_b(f):\;  \ f\in \mbox{Eq}(Q),\ \textrm{spherically symmetric with}\ E_{pot}(f)=E_{pot}(Q) \right\},\label{minTb} \ee
where
\be T_b(f)=\int_{\R^6} \left(  \frac{|v|^2}{2}+b\chi(x) x\cdot v \right) f(x,v) dx dv\label{def_Tb}\ee
and we claim the following Proposition.
\begin{proposition}\label{lemmasimilar} Let $Q$ be a steady state of \eqref{vm} in the pure Manev case which minimizes $\eqref{def-K}$. Then there exists $b^\ast>0$ such that the following holds. For all $b\in [0,b^\ast]$, the variational problem \eqref{minTb} has at least one minimizer. Moreover, there exists a family of minimizers $Q_b$ of \eqref{minTb}, taking the form
$$ Q_b(x,v)=F_{Q_b}\left(\frac{|v|^2}{2}+b\chi(x) x\cdot v +\nu_b \phi_{Q_b}(x)\right),$$
where $\nu_b$ is a positive constant and $\chi(x)$ has been defined above, and such that, as $b\rightarrow 0$, we have the convergences $\nu_b\rightarrow 1$ and $Q_b\rightarrow Q_0=Q$ in $\mathcal{E}_j$. Here, the function $Q_b$ has its support in $\{(x,v)\in \R^6, |x|<r_\chi\}$,
 the function $\phi_{Q_b}$ belongs to $\mathcal C^1$ and the function $F_{Q_b}$ is a continuous nonnegative function on $\RR$, which is $\mathcal C^1$ on $]-\infty,e_b[$ for some $e_b<0$ and vanishes on $[e_b,+\infty[$. 
 \end{proposition}
This result will be proved in the sections below. Now, using this Proposition \ref{lemmasimilar}, we end the proof of Theorem \ref{thm6}. To obtain the desired form \fref{desired} we first rescale the function $Q_b$ given by this proposition as follows:
$$\bar{Q}_b(x,v)=Q_b\left(x, \frac{v}{\nu_b}\right)$$
and set $\tilde{b}=\nu_b b$. This ensures that $\bar{Q}_b$ is a function of 
$$\frac{|v|^2}{2}+\tilde{b} \chi(x) x\cdot v + \phi_{\bar{Q}_b}(x).$$
Denoting back $\bar{Q}_b$ and $\tilde b$ by $Q_b$ and $b$ respectively, we can conclude that we have constructed a function $Q_b$ of the desired form. To prove that this function is a solution to \fref{jeans}, we need the continuity of $Q_b$ on $\RR^6$ and its $\mathcal C^1$ regularity in the interior of its support. This regularity can be deduced in a similar way as for $Q$, see Section \ref{subsect2}, Step 3.
The proof of Theorem \ref{thm6} is complete. It remains to prove Proposition \ref{lemmasimilar}. For his purpose, we need some tools which we introduce in the following subsection.

\subsection{Reduction to a functional of a modified microscopic energy}\label{subsectrearrang}
Let us first define 
\be
\label{defphirad}
\Phi_{rad}=\{\phi_f,\ f\in \mathcal{E}_{j,rad}\},\ \textrm{with} \ \mathcal{E}_{j,rad}=\{f\in \mathcal{E}_j,\ f \  \textrm{spherically symmetric}\}.
\ee
From Lemma \ref{propC1}, we deduce that there exists a constant $C>0$ such that, for all $\phi_f \in \Phi_{rad}$, we have 
\be \phi_f(r)\geq -\frac{C}{r^{3/2}}\|f\|_{\calE_j}.\label{controlphi} \ee
Moreover, by interpolation and using (H2), we have
$$\|\rho_f\|_{L^1}+\|\rho_f\|_{L^\alpha}\leq C\|f\|_{\calE_j},\quad \alpha=\frac{5p-3}{3p-1}.$$
Therefore, from Hardy-Littlewood-Sobolev, one deduces that $\phi_f\in L^k$, for all $\frac{3}{2}<k\leq \frac{3(5p-3)}{4p}$, and in particular, since $p>3$, one has
\be
\label{L3}
\|\phi_f\|_{L^3}\leq C\|f\|_{\calE_j}.
\ee

\bs
As we said, we shall construct $Q_b$ as a minimizer of the functional $T_b$ defined by \fref{def_Tb} under equimeasurability constraint. Similarly as in \cite{LM3,LMR-inv}, the key tool to study this variational problem is the symmetrization with respect to the microscopic energy $\frac{|v|^2}{2}+b\chi(x) x\cdot v + \phi(x)$. Before defining this symmetrization, we need to introduce and study the Jacobian associated with this change of variable.
\begin{lemma}[Definition and properties of the Jacobian $a_{b,\phi}$]\label{jacobian}
Let $\phi \in \Phi_{rad}\backslash\{0\}$ and $b\in \R_+$. Defining the Jacobian function $a_{b,\phi} : \R_{-}^{\ast}\rightarrow \R^{+}$ as
$$a_{b,\phi}(e)={\rm meas}\left\{(x,v) \in \R^6,\ \ \frac{|v|^2}{2}+b\chi(x) x\cdot v +\phi(x)<e\right\}, \qquad \forall e<0,$$
where $"{\rm meas}"$ stands for the Lebesgue measure on $\RR^6$, we have the following properties.\\
(i) The Jacobian $a_{b,\phi}$ is given by the explicit formula
\be\forall\ e<0, \ \  a_{b,\phi}(e)=\frac{32\pi^2 \sqrt{2}}{3}\int_{0}^{+\infty}\left(e-\phi(r)+\frac{(b\chi(r)r)^2}{2}\right)^{3/2}_{+}r^2 d r. \label{expression_a}\ee
(ii) Let $e_{b,\phi}={\rm inf \,ess} \left[ \phi(r)- \frac{(b\chi(r)r)^2}{2}\right] \in \R_{-}^{\ast}\cup\{-\infty\}$. Then $a_{b,\phi}(e)=0$ for all $e\leq e_{b,\phi}$ and $a_{b,\phi}$ is a strictly increasing $\mathcal{C}^1$ diffeomorphim from $(e_{b,\phi},0)$ onto $\R_{+}^{\ast}$. We will denote by $a_{b,\phi}^{-1}$ the inverse function of $a_{b,\phi}$.\\
(iii) Let $(f_n)$ be a bounded sequence in  $\mathcal{E}_{j,rad}\backslash \{0\}$ and let $(e_n,b_n)$ be a sequence in $\R_{-}^{\ast}\times\R^+ $ and assume that there exist $f \in \mathcal{E}_{j,rad}\backslash \{0\}$, $e\in \R^{-}\cup\{-\infty\}$ and $b\in\R^+$ such that
$$f_n \rightharpoonup f \ \textrm{in} \ L^p(\R^6),\ \ e_n\rightarrow eÊ\ \textrm{and} \ b_n\rightarrow b.$$
Then, by denoting $a_{b,\phi_f}(-\infty)=0$ and $a_{b,\phi_f}(0)=+\infty$, we have\\
$$a_{b_n,\phi_{f_n}}(e_n) \rightarrow a_{b,\phi_f}(e)\quad \mbox{and}\quad \forall s>0,\quad a_{b_n,\phi_{f_n}}^{-1}(s) \rightarrow a_{b,\phi_f}^{-1}(s).$$
\end{lemma}

\begin{proof} {\em Proof of (i)}. We remark that 
$$\frac{|v|^2}{2}+b\chi(x) x\cdot v=\frac{1}{2} |v+b\chi(x) x|^2-\frac{(b\chi(x)|x|)^2}{2}.$$
Hence, by performing the change of variable $w=v+b\chi(x) x$ with respect of the variable $v$, and passing by the spherical coordinate we find
$$a_{b,\phi}(e)=(4\pi)^2 \int r^2dr \int \un_{\{\frac{u^2}{2}+\phi(r)-\frac{(b\chi(r)r)^2}{2}<e\}}u^2 d u.$$
Formula \fref{expression_a} follows. Since $\phi\in  \Phi_{rad}\backslash\{0\}$, we have both the control \fref{controlphi} and the fact that the set $\{x\in\R^3,\ \phi(x)<e\}$ is bounded. These two properties ensure that the integral in \fref{expression_a} is finite.

\ms
\ni
 {\em Proof of (ii)}. By dominated convergence, we deduce that $e\mapsto a_{b,\phi}(e)$ is $\mathcal{C}^1$ on $\R_{-}^{\ast}$ with 
$$a_{b,\phi}'(e)=16\pi^2 \sqrt{2}\int_{0}^{+\infty}\left(e-\phi(r)+\frac{(b\chi(r)r)^2}{2}\right)^{1/2}_{+}r^2 d r \geq 0.$$
For $e_{b,\phi}<e<0$, we have clearly $a_{b,\phi}'(e)>0$. Let us prove that $a_{b,\phi}(e)$ converges to $+\infty$ as $e\rightarrow 0$. We observe that, for $f\in\mathcal{E}_{j,rad}\backslash \{0\}$ there exists $R>0$ such that $\|f_R \|_{L^1}=\frac 12\|f\|_{L^1}$, where $f_R=f \un_{|x|<R}$. Thus 
one can prove that 
\be |\phi_f(r)|\geq |\phi_{f_R}(r)| \sim \frac{\|f_R \|_{L^1}}{2\pi^2 r^2} \ \ \textrm{as} \ r \rightarrow +\infty, \label{controlinfphi}\ee
which gives 
$$a_{b,\phi}(e)\geq C\int_0^{+\infty}\left(e+\frac{C}{1+r^2}\right)^{3/2}r^2dr.$$
Hence, $a_{b,\phi}(e)\to +\infty$ as $e\to 0$. Since, we have clearly $a_{b,\phi}(e_{b,\phi})=0$, item {\em (ii)} is proved.

\ms
\ni
 {\em Proof of (iii)}.  The sequence $(f_n)$ is bounded in $\mathcal{E}_{j,rad}$. Then, from Lemma \ref{propC2} in Appendix B, up to extraction of a subsequence, we have in particular
$$ \phi_{f_n} \rightarrow \phi_f \ \textrm{almost everywhere}.$$
From \eqref{controlphi} and the boundedness of $(f_n)$ in $\mathcal{E}_{j,rad}$, we have 
$$\phi_{f_n} (r)\geq -\frac{C}{r^{3/2}}.$$
Thus
$$e_n- \phi_{f_n}(r)+\frac{(b_n\chi(r)r)^2}{2}\leq e_n+\frac{C}{r^{3/2}}+\frac{(b_n\chi(r)r)^2}{2} \rightarrow e_n$$
when $r$ converges to the infinity. Since $e_n\to e$ as $n\to +\infty$, this implies that, for $e\in [-\infty,0)$, the function 
$$r\mapsto \left(e_n- \phi_{f_n}(r)+\frac{(b_n\chi(r)r)^2}{2}\right)_+$$ is uniformly compactly supported. Therefore, by dominated convergence, we deduce that 
$$a_{b_n,\phi_{f_n}}(e_n) \rightarrow a_{b,\phi_f}(e).$$
Let us now treat the case $e=0$. Remark first that for all $n \in \NN$, we have
$$a_{b_n,\phi_n}(e_n) \geq a_{0,\phi_n}(e_n),$$
and thus it is sufficient to prove that $a_{0,\phi_n}(e_n)$ converges to $a_{b,\phi}(0)=+\infty$ as $n \rightarrow +\infty$.
Let $M>0$ be an arbitrary constant. We know that 
\be \forall x\in \RR^3, \qquad\left| \phi_f(x) \right| \geq \frac{C_f}{1+|x|^2}.\label{controlCf}\ee
Denote $\Omega_n=\left\{x\in \RR^3, \, | \phi_{f_n}(x) | <  \frac{C_f}{2(1+|x|^2)} \right\} $ and let $e_0<0$ be such that 
$$ \int_{\RR^3}  \left(e_0+\frac{C_f}{2(1+|x|^2)}\right)^{3/2}_+ dx > \frac{3M}{4\pi \sqrt{2}}.$$
For $n$ large enough, we have $e_n>e_0$ and thus
\be \begin{array}{rcl}
\dis a_{0,\phi_n}(e_n) & \geq & \dis  \frac{ 8 \pi\sqrt{2}}{3} \int_{\RR^3\backslash \Omega_n}  \left(e_0+\frac{C_f}{2(1+|x|^2)}\right)^{3/2}_+ dx\\
 & \geq & \dis 2M - \frac{ 8 \pi\sqrt{2}}{3} \int_{ \Omega_n}  \left(e_0+\frac{C_f}{2(1+|x|^2)}\right)^{3/2}_+ dx.
\end{array} \ee
To prove that the second term converges to $0$ as $n \rightarrow +\infty$, we remark that the set of integration of this term has the form $\Omega_n \cap B(0,R)$ with $R>0$ independent of $n$. Now, from \eqref{controlCf} and from the definition of $\Omega_n$,
$$\left\| \phi_{f_n}(x)-\phi_{f}(x) \right\|_{L^3(\RR^3)}^3 \geq  \int_{\Omega_n} \left( \frac{C_f}{2(1+|x|^2)} \right)^3 dx\geq  \int_{\Omega_n\cap B(0,R)} \left( \frac{C_f}{2(1+|x|^2)} \right)^3 dx.$$
Since $\phi_{f_n}$ converges to $\phi_{f}$ in $L^3(\RR^3)$ by Lemma \ref{propC2}, we deduce that the measure of the set $\Omega_n \cap B(0,R)$ converges to $0$, which implies that the integral
$$\int_{ \Omega_n\cap B(0,R)}  \left(e_0+\frac{C_f}{2(1+|x|^2)}\right)^{3/2}_+ dx$$
converges to $0$ as $n\rightarrow +\infty$. Thus, for $n$ large enough, $a_{b_n,\phi_n}(e_n) \geq a_{0,\phi_n}(e_n) \geq M.$ which concludes the proof of the convergence of $a_{b_n,\phi_n}(e_n)$. 

Now, we shall prove that for all $s\in \RR_+^*$, we have $a_{b_n,\phi_{f_n}}^{-1}(s) \rightarrow a_{b,\phi_f}^{-1}(s)$. Let $e_n:=a_{b_n,\phi_{f_n}}^{-1}(s)$. We know from the above result that, if $e_n$ converges to $e\in [-\infty,0]$, then 
$$s=a_{b_n,\phi_{f_n}}(e_n) \rightarrow a_{b,\phi_f}(e).$$
Hence, the sequence $(e_n)$ converges to $e=a_{b,\phi_f}^{-1}(s)$. The proof of Lemma \ref{jacobian} is complete.
\end{proof}

\bs
Now, we are ready to construct our symmetrization of $f$ with respect to a given microscopic energy $\frac{|v|^2}{2}+b\chi(x) x\cdot v +\phi(x)$. To that purpose, we first recall that the Schwarz symmetrization  $Q^\ast$ of the function $Q$ is the unique nonincreasing function on $\R^+$ such that $$\forall \lambda>0,\qquad \mu_{Q^*}(\lambda)=\mu_Q(\lambda),$$
where $$\mu_{Q^*}(\lambda) ={\rm meas}\{s\in \R^+\ :\ Q^\ast(s) > \lambda\},\ \ \mu_Q(\lambda)={\rm meas}\left\{(x,v)\in \RR^6:\,Q(x,v)>\lambda\right\},$$
and where the notation $"{\rm meas}"$ stands for the Lebesgue measure respectively on $\RR^+$ and $\RR^6$.
Note that $Q^*$ is compactly supported and continuous (since $Q$ is compactly supported and continuous). We shall denote 
\be
\label{rstar}
r_\ast=\min\{r\in \R^+: Q^\ast(r)=0  \}.
\ee
\sloppy
\begin{lemma}[Rearrangement with respect to the microscopic energy]\label{rearrangement}
 Let $\phi \in \Phi_{rad}\backslash\{0\}$ and $b\geq 0$. We denote by $Q^{\ast b,\phi}$ the nonincreasing continuous function of the microscopic energy defined by
$$Q^{\ast b,\phi}(x,v)=\left\{\begin{array} {lcl}  Q^\ast \circ a_{b,\phi} \left(\frac{|v|^2}{2}+b\chi(x) x\cdot v +\phi(x)\right) & \textrm{if} & \frac{|v|^2}{2}+b\chi(x) x\cdot v +\phi(x)<0, \\ 0 & \textrm{if} & \frac{|v|^2}{2}+b\chi(x) x\cdot v +\phi(x) \geq 0. 
\end{array} \right.$$
Then the following holds.\\
(i) The function $Q^{\ast b,\phi}$ is compactly supported and $${\rm Supp}(Q^{\ast b,\phi})=\left\{\ds \frac{|v|^2}{2}+b\chi(x) x\cdot v +\phi(x)<a_{b,\phi}^{-1}(r_\ast)\right\},$$
where $r_*$ is defined by \fref{rstar}.\\
(ii) We have $Q^{\ast b,\phi}\in {\rm Eq}(Q)$ and
\be  \int_{\R^6} |v|^2 Q^{\ast b,\phi} (x,v)dxdv \leq C( \| \phi \|_{L^{3}}^2+b^2). \label{control_kinetic}\ee
(iii) Let $(f_n)$ be a bounded sequence of  $\mathcal{E}_{j,rad}\backslash \{0\}$ and let $(b_n)$ be a sequence of $\R^+ $ such that $f_n \rightharpoonup f\neq 0$ in $L^p(\R^6)$ and $b_n\rightarrow b$. Then,
$$Q^{\ast b_n,\phi_{f_n}} \rightarrow Q^{\ast b,\phi_f} \ \ \textrm{in} \ \ L^1(\R^6)\cap L^p(\R^6).$$
(iv) For all $f\in {\rm Eq}(Q)$, spherically symmetric, and for all $\nu>0$, we have
\be \int \left(\frac{|v|^2}{2}+b\chi(x) x\cdot v +\nu \phi_{f}(x)\right) \left(Q^{\ast b,\nu \phi_{f}}(x,v)-f(x,v)\right) dxdv \leq 
0   \label{ineg_rearrang}\ee
with equality if, and only if, $f=Q^{\ast b,\nu \phi_{f}}$.
\end{lemma} 

\begin{proof} We first remark that property {\em (i)} is a direct consequence of the definition of $Q^{\ast b,\phi}$.\\
{\em Proof of (ii).} Recall that, for all $\lambda>0$,
\be \mu_{Q^{\ast b,\phi}}={\rm meas}\left\{(x,v)\in\R^6, \quad Q^{\ast b,\phi}(x,v)>\lambda\right\}.\ee
If $\lambda\geq Q^\ast(0)=\|Q\|_{L^\infty}$, we clearly have $\mu_{Q^{\ast b,\phi}}(\lambda)=0=\mu_{Q}(\lambda)$. If $\lambda<Q^\ast(0)$, then we have
$$\begin{array}{rcl} 
\mu_{Q^{\ast b,\phi}}(\lambda) & = & {\rm meas}\left\{(x,v)\in\R^6, a_{b,\phi}\left(\frac{|v|^2}{2}+b\chi(x) x\cdot v +\phi(x)\right)<\sup\{r,Q^\ast(r)>\lambda\}\right\} \\ & = & \sup\{r,Q^\ast(r)>\lambda\}=\mu_{Q}(\lambda).
\end{array}$$
Thus the functions $Q^{\ast b,\phi}$ and $Q$ are equimeasurable. To estimate its kinetic energy, we remark that
\be T_b(Q^{\ast b,\phi}) \geq \int \left(\frac{|v|^2}{2}- b R_\chi |v|\right) Q^{\ast b,\phi}  
\geq \frac{1}{4} \int |v|^2Q^{\ast b,\phi}  -b^2R_\chi^2 \|Q\|_{L^1}. \label{control_Tb1}\ee
Moreover, from the definition of $Q^{\ast b,\phi}$, one deduces that
$$\begin{array}{rcl}
 T_b(Q^{\ast b,\phi}) & = &\dis  \int  \left(\frac{|v|^2}{2}+b\chi(x) x\cdot v +\phi(x)\right)Q^{\ast b,\phi} (x,v)dx dv 
 -\int \phi(x)Q^{\ast b,\phi}(x,v) dxdv \\
  & \leq & \dis -\int \phi(x)Q^{\ast b,\phi}(x,v) dxdv \leq \| \phi \|_{L^{3}} \| \rho_{Q^{\ast b,\phi}} \|_{L^{3/2}} \\
  & \leq & \dis C\| \phi \|_{L^3} \| Q^{\ast b,\phi}\|_{L^3}^{1/2} \| |v|^2 Q^{\ast b,\phi} \|_{L^1}^{1/2} \leq C \| \phi \|_{L^{3}} \| |v|^2 Q^{\ast b,\phi} \|_{L^1}^{1/2},
\end{array}$$
where we used an interpolation inequality and $\|Q^{\ast b,\phi}\|_{L^3}=\|Q\|_{L^3}$. Combining this with \eqref{control_Tb1} gives the control of the kinetic energy \eqref{control_kinetic}.

\ms
\ni
{\em Proof of (iii).} From the continuity of $Q^\ast$ and from Lemma \ref{jacobian} {\em (ii)} and {\em (iii)}, we clearly have, for any sequence $e_n\rightarrow e$,
$$Q^{\ast}\circ a_{b_n,\phi_{f_n}}(e_n)\rightarrow Q^{\ast}\circ a_{b,\phi_f}(e).$$
Moreover, by Lemma \ref{propC2}, up to a subsequence, $\phi_{f_n}\rightarrow \phi_f$ almost everywhere in $\R^6$. Denoting 
$$e_n(x,v)=\frac{|v|^2}{2}+b_n\chi(x) x\cdot v + \phi_{f_n}(x),\qquad e(x,v)=\frac{|v|^2}{2}+b\chi(x) x\cdot v + \phi_{f}(x),$$
we deduce that 
$$\textrm{for a.e.} \quad (x,v)\in\R^6,\ e_n(x,v) \rightarrow  e(x,v).$$
Thus $Q^{\ast b_n,\phi_n}$ converges to $Q^{\ast b,\phi}$ almost everywhere in $\R^6$ and the equimeasurability of $Q^{\ast b_n,\phi_n}$ and $Q^{\ast b,\phi}$ gives the convergence in $L^1\cap L^p$.

\ms
\ni
{\em Proof of (iv).} Let $f\in {\rm Eq}(Q)$ be spherically symmetric and let $\nu>0$. We have $\phi:=\nu \phi_f=\phi_{\nu f} \in \Phi_{rad}\backslash\{0\}$. We denote $\bar{f}=Q^{\ast b, \phi}$ and we use the layer cake representation
$$f(x,v)=\int_{t=0}^{\|f\|_\infty} \mathbf{1}_{t<f(x,v)} dt.$$
Then from Fubini's theorem,\\
$\dis \int_{\R^6} \left(\frac{|v|^2}{2}+b\chi(x) x\cdot v +\phi(x)\right) \left(f(x,v)-\bar{f}(x,v)\right)dxdv$
$$ \begin{array}{cl}
  = & \dis \int_{t=0}^{\|f\|_\infty} dt \int_{\R^6} \left(\mathbf{1}_{t<f(x,v)} - \mathbf{1}_{t<\bar{f}(x,v)} \right) \left(\frac{|v|^2}{2}+b\chi(x) x\cdot v +\phi \right) dx dv \\ 
  = & \dis \int_{t=0}^{\|f\|_\infty} dt  \int_{\R^6} \left(\mathbf{1}_{\bar{f}(x,v) \leq t<f(x,v)} - \mathbf{1}_{f(x,v)\leq t<\bar{f}(x,v)} \right) \left(\frac{|v|^2}{2}+b\chi(x) x\cdot v +\phi \right) dx dv \\
  = & \dis \int_{t=0}^{\|f\|_\infty} dt \left( \int_{S_1(t)} \left(\frac{|v|^2}{2}+b\chi(x) x\cdot v +\phi \right) dx dv - \int_{S_2(t)} \left(\frac{|v|^2}{2}+b\chi(x) x\cdot v +\phi \right) dx dv \right) ,
  \end{array} $$
  with $$ S_1(t)=\{\bar{f}(x,v) \leq t<f(x,v)\}, \ \ \ S_2(t)=\{f(x,v)\leq t<\bar{f}(x,v)\}.$$
Now, from the equimeasurability of $f$ and $\bar{f}$, we have 
$$ \forall t>0,\qquad {\rm meas}(S_1(t))={\rm meas}(S_2(t))$$
and, since $\bar{f}$ is a nonincreasing function of $\frac{|v|^2}{2}+b\chi(x) x\cdot v +\phi(x)$, 
\bee a_2(t)&=&\sup_{(x,v)\in S_2(t)} \left\{\frac{|v|^2}{2}+b\chi(x) x\cdot v +\phi(x)\right\}\\
&\leq& \inf_{(x,v)\in S_1(t)} \left\{\frac{|v|^2}{2}+b\chi(x) x\cdot v +\phi(x)\right\}=a_1(t).
\eee
Thus 
$$\begin{array} {rcl}
\dis \int_{S_2(t)} \left(\frac{|v|^2}{2}+b\chi(x) x\cdot v +\phi(x) \right) dx dv & \leq & {\rm meas}(S_2(t)) a_2(t) \leq {\rm meas}(S_1(t)) a_1(t) \\ & \leq &\dis  \int_{S_1(t)} \left(\frac{|v|^2}{2}+b\chi(x) x\cdot v +\phi(x) \right) dx dv,
\end{array}$$
which yields \fref{ineg_rearrang}. In the case of equality in this above chain of inequalities, it is easy to prove that $f=\bar{f}$, see for instance \cite{LM3,LMR-inv}.
\end{proof}

\subsection{Existence of self-similar solutions}\label{subsectselfsim} The goal of this subsection is to prove Proposition \ref{lemmasimilar}. 

\ms
\ni
{\em Step 1: uniform bounds.} Let $0<b\leq 1$ be given. In this step, we prove that $T_b$ is finite and that there exists $C^*>0$, independent of $b$, such that every minimizing sequence $(f_n^b)_{n\in \NN}$ of \fref{minTb} satisfies, for $n$ large enough,
\be
\label{estiC*}
\|f_n^b\|_{\calE_j}\leq C^*.
\ee
For all $(x,v)\in\R^6$, we have 
$$\frac{|v|^2}{2}+b\chi(x) x\cdot v \geq \frac{|v|^2}{2}-bR_\chi |v| \geq \frac{|v|^2}{4}-b^2R_\chi^2$$ and so, for all $f\in {\rm Eq}(Q)$, we have
\be T_b(f) \geq \frac{1}{4}\int |v|^2 f- b^2R_\chi^2\|Q\|_{L^1}. \label{controlTb}\ee
This shows that $T_b>-\infty$. Moreover, if $(f_n^b)_{n\in \NN}$ is a minimizing sequence of the variational problem \fref{minTb}, then for $n$ large enough we have
$$ T_b(f_n^b)\leq 1+T_b\leq 1+T_b(Q)\leq 1+C\int |v|^2 Q.$$
This, combined with \fref{controlTb} yields the existence of $C^*$.

\bs
\ni
{\em  Step 2.} For all $b\in \RR_+$, let $(f_n^b)_{n\in \NN}$ be a minimizing sequence for \fref{minTb}. In this step, we show that there exists a sequence $(\nu_n^b)$ of positive numbers such that $(Q^{\ast b,\nu_n^b \phi_{f_n^b}})$, defined by Lemma \ref{rearrangement}, is also a minimizing sequence of \eqref{minTb}. The interest of this new minimizing sequence is its compactness property, as it will be proved in the third step.
\begin{lemma}\label{lemma_epot} There exists $b^\ast>0$ such that the following holds true. For all $b\in[0,b^\ast]$ and for all minimizing sequences $(f_n^b)_{n\in \mathbb{N}}$ of the variational problem \eqref{minTb}, there exist $0<\nu^-<\nu^+$ and a sequence of positive numbers $(\nu_n^b)_{n\in\mathbb{N}}$ in $[\nu^-,\nu^+]$ such that, up to a subsequence, we have $E_{pot}(Q^{\ast b,\nu_n^b \phi_{f_n^b}})=E_{pot}(Q)$.  Moreover, we have $T_b(Q^{\ast b,\nu_n^b \phi_{f_n^b}})\leq T_b(f_n^b)$ with equality if and only if $Q^{\ast b,\nu_n^b \phi_{f_n^b}}=f_n^b$.
\end{lemma} 

\ni
{\em Proof of Lemma \pref{lemma_epot}.} Let $b>0$ be given and consider a minimizing sequence $(f_n^b)_{n\in \mathbb{N}}$ of the variational problem \eqref{minTb}. From Step 1, we know that (for $n$ large enough), this sequence satisfies the bound \fref{estiC*}. 

We first observe that $\nu\mapsto E_{pot}(Q^{\ast b,\nu \phi_{f_n^b}})$ is continuous on $\R_+^\ast$. Indeed, by Lemma \ref{rearrangement} {\em (iii)}, we know that $\nu\mapsto  Q^{\ast b,\nu \phi_{f_n^b}}$ is continuous from $\RR_+^\ast$ to $L^1(\R^6)\cap L^p(\R^6)$.
Hence, from the kinetic control \fref{control_kinetic} and Lemma \ref{propC2} of Appendix B, one deduces the continuity of $\nu\mapsto E_{pot}(Q^{\ast b,\nu \phi_{f_n^b}})$.

We claim now that, for $b$ small enough, there exist $0<\nu^-<\nu^+$ such that, up to a subsequence with respect of $n$, we have
\be
\label{claimm}E_{pot}(Q^{\ast b,\nu^- \phi_{f_n^b}})<E_{pot}(Q)<E_{pot}(Q^{\ast b,\nu^+ \phi_{f_n^b}}).
\ee
Since $Q^{\ast b,\nu \phi_{f_n^b}}\in {\rm Eq}(Q)$, we have, by \fref{ineg-inter-M},
\be 0<E_{pot}(Q^{\ast b,\nu \phi_{f_n^b}})\leq C\int_{\R^6} |v|^2 Q^{\ast b,\nu \phi_{f_n^b}}(x,v)dxdv. \label{boundbykinetic}\ee
Furthermore, the control of the kinetic energy \eqref{control_kinetic}, together with \fref{L3}, gives 
$$  \int |v|^2 Q^{\ast b,\nu \phi_{f_n^b}} \leq C( \nu^2 \|f_n^b\|_{\calE_j}^2+b^2)\leq C( (C^*\nu)^2+b^2),$$
where we also used \fref{estiC*}. Hence from \fref{boundbykinetic}, one deduces that there exists $b^\ast_1\in (0,1]$ and $\nu^->0$ such that for all $b\in [0,b^\ast_1]$ and for all $n$, we have
$$E_{pot}(Q^{\ast b,\nu^- \phi_{f_n^b}}) <E_{pot}(Q).$$
Note that $b_1^*$ depends only on $Q$, and does not depend on the sequence $(f_n^b)$.

Let us now prove the second part of the claim \fref{claimm}. Since the sequence $(f_n^b)_{n\in\NN}$ is bounded in $\calE_j$, Lemma \ref{propC2} of Appendix B implies that there exists $f_b\in \mathcal{E}_{j,rad}$ such that up to a subsequence, as $n\rightarrow +\infty$,
$$f_n^b  \rightharpoonup f_b \  \textrm{in} \ L^p(\R^6), \quad  E_{pot}(f_b)=\lim_{n\rightarrow +\infty} E_{pot}(f_n^b)=E_{pot}(Q).$$
In particular, $f_b\neq 0$ and, by Lemma \ref{rearrangement} {\em (iii)}, for all $\nu>0$ we have
$$Q^{\ast b,\nu \phi_{f_n^b}} \rightarrow Q^{\ast b,\nu \phi_{f_b}} \ \ \textrm{in}\  L^1(\RR^6)\cap L^p(\RR^6) \, \mbox{ as } n\rightarrow +\infty.$$
Thus, from the kinetic control \eqref{control_kinetic}, $(Q^{\ast b,\nu \phi_{f_n^b}} )$ is bounded in $\mathcal{E}_j$ and spherically symmetric, which implies that $E_{pot}(Q^{\ast b,\nu \phi_{f_n^b}} )$ converges to $E_{pot}(Q^{\ast b,\nu \phi_{f_b}} )$ as $n\rightarrow +\infty$. Consequently, to prove the claim \fref{claimm}, it is sufficient to show that, if $b$ is small enough, there exist $\nu^+$ such that $E_{pot}(Q^{\ast b,\nu^+ \phi_{f_b}} )>E_{pot}(Q)$. This result will a consequence of the following lemma, which is proved later.
\begin{lemma}\label{trouve-bast} There exists $b^\ast_2>0$ such that the following holds true. For all $f\in \mathcal{E}_{j,rad}$  satisfying
\be E_{pot}(f)=E_{pot}(Q), \quad \| f\|_{L^1}\leq \| Q\|_{L^1},\quad   \|j( f )\|_{L^1}\leq \| j(Q)\|_{L^1}\label{hyp1-bast} \ee
and
\be T_b(f ) +\frac{(b R_\chi)^2}{2}\| f \|_{L^1} \leq T_b + \frac{(b R_\chi)^2}{2} \| Q\|_{L^1}, \label{hyp2-bast}\ee
for some $b\in [0,b^\ast_2]$, we have
\be  \limsup_{\nu\rightarrow +\infty} E_{pot}(Q^{\ast b,\nu \phi_{f}} ) > E_{pot}(Q). \ee
\end{lemma}
Before proving Lemma \ref{trouve-bast}, let us use it to end the proof of Lemma \ref{lemma_epot} and Proposition \ref{lemmasimilar}. Let us check that $f_b$ satisfies the assumptions of Lemma \ref{trouve-bast}, for $b\leq b^*_2$. Note that $b^\ast_2$ given by this lemma is independent of the function $f_b$. 
First, $f_b$ satisfies Assumption \eqref{hyp1-bast} because of the weak convergence of $(f_n^b)$ to $(f_b)$ and of the strong convergence of the potential energies. To prove that $f_b$ satisfies Assumption \eqref{hyp2-bast}, we remark that
$$ \forall (x,v)\in \RR^6,\ \frac{|v|^2}{2} + b \chi(x) x\cdot v + \frac{(b R_\chi)^2}{2}= \frac{|v+b\chi(x)x |^2}{2}+\frac{(b R_\chi)^2-(b\chi(x)|x|)^2}{2}\geq 0.$$
Hence, by lower semicontinuity, one has
$$T_b(f_b)+\frac{(b R_\chi)^2}{2}\|f_b\|_{L^1}\leq \liminf_{n\to +\infty}\left(T_b(f^b_n)+\frac{(b R_\chi)^2}{2}\|f^b_n\|_{L^1}\right)=T_b+\frac{(b R_\chi)^2}{2}\|Q\|_{L^1}.$$
Therefore, we may apply Lemma \ref{trouve-bast} and get the existence of $\nu^+$ such that, for $b\leq b_2^\ast$, 
$$E_{pot}(Q^{\ast b,\nu^+ \phi_{f_b}} ) > E_{pot}(Q).$$
Hence the claim \fref{claimm} holds true for all $0\leq b\leq b^*=\min(b^*_1,b^*_2)$. Note that $b^*$ is independent of the sequence $(f_n^b)$. One can chose $\nu^b_n\in [\nu^-,\nu^+]$ such that for all $n$, we have  $E_{pot}(Q^{\ast b,\nu^b_n \phi_{f_n^b}})=E_{pot}(Q)$. 

\bs
Now, it remains to show the second part of Lemma \ref{lemma_epot}. We have
$$\begin{array} {rcl}
T_b(Q^{\ast b,\nu^b_n \phi_{f^b_n}}) &= & \dis \int \left(\frac{|v|^2}{2}+b\chi(x) x\cdot v +\nu^b_n \phi_{f^b_n}(x)\right) Q^{\ast b,\nu^b_n \phi_{f^b_n}}(x,v)dxdv \\ & & \dis -\nu^b_n \int \phi_{f^b_n}(x)Q^{\ast b,\nu^b_n \phi_{f^b_n}}(x,v)dxdv \\
 & \leq &\dis T_b(f^b_n) + \nu_n^b \int \phi_{f^b_n}(x)f^b_n(x,v)dxdv  -\nu^b_n \int \phi_{f^b_n}(x)Q^{\ast b,\nu^b_n \phi_{f^b_n}}(x,v)dxdv,
\end{array}$$
from the inequality $\eqref{ineg_rearrang}$. Observing that
$$-\int \phi_{f^b_n}(x)Q^{\ast b,\nu^b_n \phi_{f^b_n}}(x,v)dxdv=\frac{1}{2}\left(E_{pot}(f^b_n)+E_{pot}(Q^{\ast b,\nu^b_n \phi_{f^b_n}})-E_{pot}(Q^{\ast b,\nu^b_n \phi_{f^b_n}}-f^b_n)\right),$$
we get
$$
T_b(Q^{\ast b,\nu^b_n \phi_{f^b_n}})  \leq  T_b(f^b_n) + \frac{\nu^b_n}{2}\left(E_{pot}(Q^{\ast b,\nu^b_n \phi_{f^b_n}})-E_{pot}(f^b_n) \right)-\frac{\nu^b_n}{2} E_{pot}(Q^{\ast b,\nu^b_n \phi_{f^b_n}}-f^b_n).
$$
Since
$$E_{pot}(Q^{\ast b,\nu^b_n \phi_{f^b_n}})=E_{pot}(f^b_n) =E_{pot}(Q),$$
we deduce that  $T_b(Q^{\ast b,\nu^b_n \phi_{f^b_n}}) \leq T_b(f^b_n)$. By Lemma \ref{rearrangement} {\em (iv)}, this inequality becomes an equality if and only if $f^b_n=Q^{\ast b,\nu^b_n \phi_{f^b_n}}$. The proof of Lemma \ref{lemma_epot} is complete.
\qed

\bs
\ni
{\em Step 3: construction of $Q_b$, minimizer of \eqref{minTb}.} Let $b\in [0,b^*]$, where $b^*$ is defined in Lemma \ref{lemma_epot}, and let $f_n^b$ be a minimizing sequence of the variational problem \fref{minTb}. Then, the sequence $(\nu_n^b)$ given by Lemma \ref{lemma_epot} lies in a compact interval $[\nu^-,\nu^+]$. Up to a subsequence, $(\nu^b_n)$ converges to some $\bar \nu_b>0$ as $n\to +\infty$.
By Lemma \ref{rearrangement}, we have
$$Q^{\ast b,\nu^b_n \phi_{f^b_n}} \rightarrow Q^{\ast b,\bar \nu_b \phi_{f_b}} \ \textrm{in}\ \ L^1\cap L^p,$$
where $f_n^b\rightharpoonup f_b$ in $L^p$.
Moreover, from the kinetic control \eqref{control_kinetic}, $(Q^{\ast b,\nu^b_n \phi_{f^b_n}} )$ is bounded in $\mathcal{E}_{j,rad}$ and thus,
$$E_{pot}(Q^{\ast b,\bar \nu_b \phi_{f_b}})=\lim_{n\rightarrow+\infty} E_{pot}(Q^{\ast b,\nu^b_n \phi_{f^b_n}})=E_{pot}(Q).$$
Let us denote $Q_b:=Q^{\ast b,\bar \nu_b \phi_{f_b}}$ and make another rearrangement. Applying Lemma \ref{lemma_epot},  there exists $\nu_b>0$ such that \\
\hspace*{5mm}(i) $E_{pot}(Q^{\ast b,\nu_b \phi_{Q_b}})=E_{pot}(Q)$. \\
\hspace*{5mm}(ii) $T_b(Q^{\ast b,\nu_b \phi_{Q_b}})\leq T_b(Q_b)$ with equality only if $Q^{\ast b,\nu_b \phi_{Q_b}}=Q_b$.\\
By lower semicontinuity, we have 
$$T_b\leq T_b(Q_b)\leq \liminf_{n\to +\infty}T_b(Q^{\ast b,\nu^b_n \phi_{f^b_n}})\leq \lim_{n\to +\infty}T_b(f_n^b)=T_b.$$ Therefore $T_b(Q^{\ast b,\nu_b \phi_{Q_b}})=T_b(Q_b)=T_b$ which implies that $Q_b=Q^{\ast b,\nu_b \phi_{Q_b}}$. In particular, $Q_b$ takes the desired form \fref{desired} and similar arguments as in Section \ref{subsect2}, Step 3, give the regularity of $Q_b$ stated in Theorem \ref{thm6}.

\bs
\ni
{\em Step 4.} 
We prove here that the above constructed sequence $(Q_b)$ converges to $Q$ in $\mathcal{E}_j$, as $b\rightarrow 0$.  Remark first that $Q_0=Q$ and then $\nu_0=1$. Indeed, we claim that $Q_0$ and $Q$ are two radially symmetric equimeasurable steady states of \fref{vm} (with $\delta=0$), which minimize \fref{def-K}, and have the same kinetic energy. This enables to apply Lemma \ref{theo2} {\em (ii)} and conclude that $Q_0=Q$. Let us prove this claim. First, since $Q$ is a steady state of \fref{vm} which minimizes \fref{def-K}, and since $Q_0$ is equimeasurable to $Q$, we have
$$K(Q_0)\geq K(Q)=1.$$
Second, $Q_0$ being a minimizer of \fref{minTb} with $b=0$, and since $E_{pot}(Q_0)=E_{pot}(Q)$, we also have
$$K(Q)\geq K(Q_0).$$
This yields $K(Q)= K(Q_0)=1$ and then $Q$ and $Q_0$ are both minimizers of \fref{def-K}. Since these functions are equimeasurable, the claim is proved.

Now, similarly as for \eqref{absurd1}, one can prove that 
$$ \limsup_{b\rightarrow 0} \int \frac{|v|^2}{2} Q_b\leq \int \frac{|v|^2}{2} Q.$$
Moreover, since $Q$ is a minimizer of \eqref{def-K}, the function $Q_b$ satisfies $\dis \int |v|^2 Q_b \geq \int |v|^2 Q$ for all $b$. Thus, we have
\be
\label{z4}
\int \frac{|v|^2}{2} Q_b\to \int \frac{|v|^2}{2} Q\quad \mbox{as }b\to 0
\ee
and the sequence $(Q_b)$ satisfies 
$$ Q_b \in \mbox{Eq}(Q) \quad \mbox{ and }\quad  \frac{\mathcal{H}(Q_b)}{\| |v|^2 Q_b \|_{L^1}} \rightarrow 0.$$
Thus, by Proposition \ref{propstab}, one deduces that we have
\be Q_b\left(\lambda_b \,x,\frac{v}{\lambda_b}\right) \rightarrow Q \ \textrm{in}\  \mathcal{E}_j,\ \  \textrm{where} \ \ \lambda_b=\left(\frac{ \| |v|^2 Q \|_{L^1}}{ \| |v|^2 Q_b \|_{L^1}}\right)^{1/2}.\ee
From \fref{z4}, we finally deduce that $\lambda_b\to 1$ and that
\be Q_b \rightarrow Q \ \textrm{in}\  \mathcal{E}_j\quad \mbox{as }b\to 0.\label{convQbQ}\ee

\bs
\ni
{\em Step 5: convergence of $(\nu_b)$ as $b \to 0$.} We recall that $Q_b$ takes the form \fref{desired}, thus satisfies the equation
\be v\cdot \na_x Q_b -\nu_b \na_x\phi_{Q_b} \cdot\na_v Q_b +b \chi(x) \left(x\cdot \na_x Q_b-v\cdot\na_v Q_b \right)-b(x\cdot v) \na_x \chi \cdot \na_v Q_b=0. \label{eqstatter} \ee
We aim to apply Lemma \ref{lemvirial}. Multiply the two last terms of \fref{eqstatter} by $x\cdot v$ and integrate on $\RR^6$. Integrations by parts give
$$ \int_{\RR^6}(x\cdot v) \, b \chi(x) \left(x\cdot \na_x Q_b-v\cdot\na_v Q_b \right) dxdv= - b \int_{\RR^6}(x\cdot v) (x \cdot \nabla \chi)  Q_b dxdv$$
and
$$ -\int_{\RR^6}b (x\cdot v)^2 \,\na_x \chi \cdot \na_v Q_b dxdv= +2b \int_{\RR^6}(x\cdot v) (x \cdot \nabla \chi)  Q_b dxdv.$$
Then, from Lemma \ref{lemvirial},
\be \label{z5} \nu_b E_{pot}(Q) - \int_{\RR^6}|v|^2 Q_b(x,v)dxdv=b \int_{\RR^6}(x\cdot v) (x \cdot \nabla \chi)  Q_b dxdv,\ee
where
 $$ \left| b \int_{\RR^6}(x\cdot v) (x \cdot \nabla \chi)  Q_b dxdv\right| \leq b R_\chi^2 \| \nabla \chi \|_{L^\infty} \frac{\|Q\|_{L^1}+\| |v|^2Q_b\|_{L^1}}{2}.$$ 
Using \fref{z4}, \fref{z5} and $\mathcal H(Q)=0$, we obtain $\nu_b \rightarrow 1$ as $b\rightarrow 0$.

\bs
\ni
{\em Step 6: choice of $r_\chi$.} Now, we seek $r_\chi$ such that, for all $b\in[0,b^\ast]$, $\mbox{Supp}( \rho_{Q_b})\subset B(0,r_\chi)$.
We have seen that Lemma \ref{rearrangement} {\em (i)} gives
\be \mbox{Supp}(Q_b)\subset \left\{(x,v),\ \nu_b \phi_{Q_b}(x)<a_{b, \nu_b \phi_{Q_b}}^{-1}(r_\ast) + \frac{(b^\ast R_\chi)^2}{2} \right\},\label{inclusion1} \ee
where $r_*$ is defined by \fref{rstar}.
Remark first that from the continuity of the function $(b,\phi)\mapsto a_{b, \phi}^{-1}(r_\ast)$ in Lemma \ref{jacobian}, we deduce
$$a_{b, \nu_b \phi_{Q_b}}^{-1}(r_\ast)\rightarrow a_{0, \phi_{Q}}^{-1}(r_\ast)<0\quad \mbox{as }b\to 0.$$
Let $b^*_3>0$ small enough such that for all $0<b<b^*_3$
$$\frac{a_{b, \nu_b \phi_{Q_b}}^{-1}(r_\ast)}{\nu_b}<\frac{a_{0, \phi_{Q}}^{-1}(r_\ast)}{2}, \qquad \frac{(b^\ast R_\chi)^2}{2\nu_b}<\frac{\left|a_{0, \phi_{Q}}^{-1}(r_\ast)\right|}{4}\quad \mbox{and}\quad \|Q_b\|_{\calE_j}\leq 2\|Q\|_{\calE_j},$$
where we recall that $Q_b\to Q$ in $\calE_j$.
Then, for $b\leq\min(b^*,b^*_3)$, \fref{inclusion1} yields
\be \mbox{Supp}(\rho_{Q_b})\subset \left \{x\in\R^3,\quad \phi_{Q_b}(x)<\frac{a_{0, \phi_{Q}}^{-1}(r_\ast)}{4} \right\}.\label{inclusion2}\ee
Moreover, by \eqref{controlphi}, the function $\phi_{Q_b}$ satisfies 
\be
\label{z9}\phi_{Q_b}(x)\geq -\frac{C\|Q_b\|_{\calE_j}}{|x|^{3/2}}\geq-\frac{2C\|Q\|_{\calE_j}}{|x|^{3/2}},
\ee
where $C$ is a universal constant. Now we set
\be
\label{choix rchi}
r_\chi = \left( \frac{8C\|Q\|_{\calE_j}}{a_{0, \phi_{Q}}^{-1}(r_\ast)}\right)^{2/3},
\ee
and observe that, as mentionned in Proposition \ref{lemmasimilar}, the constant $r_\chi$ depends only on $Q$ and $\phi_Q$.
From \fref{inclusion2} and \fref{z9}, renoting $b^*=\min(b^*,b^*_3)$, we deduce that, for $0\leq b\leq b^*$,
$$ \mbox{Supp}(\rho_{Q_b})\subset \left \{x\in\R^3,\  |x|< r_\chi \right\}.$$
Finally we have just to prove the Lemme \pref{trouve-bast} to complete the proof of Proposition \ref{lemmasimilar}.
\qed

\bs
\ni
{\em Proof of Lemma \pref{trouve-bast}.} We proceed by contradiction. We suppose that there exists a sequence $b_k$ going to 0 as $k\to +\infty$ and a sequence $(f_k)$  such that, for all $k$, the function $f_k\in \mathcal{E}_j$ satisfies \eqref{hyp1-bast}, \eqref{hyp2-bast} and
\be \limsup_{\nu\rightarrow +\infty} E_{pot}(Q^{\ast b_k,\nu \phi_{f_k}} ) \leq E_{pot}(Q). \label{contr-epot} \ee
From the Cauchy-Schwarz inequality, we have for all $\nu >0$
\be -\int_{\R^6} \phi_{f_k} Q^{\ast b_k,\nu \phi_{f_k}} \leq E_{pot}(Q^{\ast b_k,\nu \phi_{f_k}} )^\frac{1}{2} E_{pot}(f_k)^\frac{1}{2}, \label{cs-nu}\ee
and thus the inequality \eqref{contr-epot} implies, for all $k$,
\be \limsup_{\nu\rightarrow +\infty} \left(-\int_{\R^6} \phi_{f_k} Q^{\ast b_k,\nu \phi_{f_k}} \right) \leq E_{pot}(Q). \label{contr-epot2}\ee 
Moreover, by Lemma \ref{rearrangement} {\em (i)}, for all $k$, we have
$$\begin{array} {rcl}
{\rm Supp}(Q^{\ast b_k,\nu \phi_{f_k}}) & \subset &\displaystyle \left\{(x,v),\quad \frac{|v|^2}{2}-b_kR_\chi |v|+\nu \phi_{f_k}(x)<a_{b_k,\nu \phi_{f_k}}^{-1}(r_\ast)\right\} \\
 & \subset &\displaystyle \left\{(x,v),\quad  \phi_{f_k}(x)<  \frac{a_{b_k,\nu \phi_{f_k}}^{-1}(r_\ast)}{\nu} + \frac{(b_k R_\chi)^2}{2\nu} \right\}.
\end{array}$$
Now, from the explicit expression of $a_{b_k,\nu \phi_{f_k}}$ and for $e=a_{b_k,\nu \phi_{f_k}}^{-1}(r_\ast)$, we get
$$r_\ast=a_{b_k,\nu \phi_{f_k}}(e) \geq a_{0,\nu \phi_{f_k}}(e)=\nu^{3/2} a_{0, \phi_{f_k}}\left(\frac{e}{\nu}\right),$$
and then,
\be
\label{www}\frac{e}{\nu} + \frac{(b_k R_\chi)^2}{2\nu}\leq  a_{0, \phi_{f_k}}^{-1}(\nu^{-3/2}r_\ast)+ \frac{(b_k R_\chi)^2}{2\nu}.
\ee
Since, as $\nu\to +\infty$, the right-hand side of \fref{www} goes to $a_{0, \phi_{f_k}}^{-1}(0)=-\| \phi_{f_k}\|_{L^\infty}\in [-\infty,0[$, we deduce that 
\be \limsup_{\nu\rightarrow +\infty} \left( -\int_{\R^6} \phi_{f_k} Q^{\ast b_k,\nu \phi_{f_k}}\right) \geq  \| \phi_{f_k}\|_{L^\infty} \| Q\|_{L^1} \geq \| \phi_{f_k}\|_{L^\infty} \| f_k\|_{L^1}\geq E_{pot}(f_k).\label{strictenfait} \ee
Hence, from \eqref{contr-epot2} and $E_{pot}(f_k)=E_{pot}(Q)$, the inequalities in \eqref{strictenfait} are all equalities. Thus the sequence $(f_k)$ satisfies for all $k$, 
\be \label{cite}\|f_k\|_{L^1}=\|Q\|_{L^1} \,\,\mbox{ and }\,\, \| \phi_{f_k}\|_{L^\infty} \| f_k\|_{L^1}=E_{pot}(Q). \ee
Now we will prove that $(f_k)$ is a minimizing sequence for the variational problem \fref{def-K} with $M_1=\|Q\|_{L^1}$ and $M_j=\|j(Q)\|_{L^1}$.
First, from \eqref{hyp2-bast} and \fref{cite}, we have  $T_{b_k}(f_k) \leq T_{b_k}$. Combining it with 
\be T_{b_k}(f_k) \geq \int \frac{|v|^2}{2} f_k- {b_k} R_\chi \int \frac{1+|v|^2}{2}f_k \geq (1-{b_k} R_\chi)\int \frac{|v|^2}{2} f_k-\frac{{b_k}R_\chi}{2}\|Q\|_{L^1}\ee
and with 
\be T_{b_k}\leq T_{b_k}(Q)=\int \frac{|v|^2}{2} Q \ee
given by the definition of $T_{b_k}$ and the radial symmetry of $Q$, we obtain
\be \limsup_{k\to +\infty} \int \frac{|v|^2}{2} f_k\leq \int \frac{|v|^2}{2} Q. \label{absurd1}\ee
Finally, using $E_{pot}(f_k)=E_{pot}(Q)$, $\calH(Q)=0$ and the interpolation inequality \fref{ineg-inter-M} which gives a lower bound for $\| |v|^2 f_k \|_{L^1}$, one gets
\be
\|f_k\|_{L^1} =  \| Q\|_{L^1},\quad  \limsup_{k\to +\infty} \|j(f_k)\|_{L^1} \leq \| j(Q) \|_{L^1}\mbox{ and }\limsup_{k\to +\infty} \frac{\mathcal{H}(f_k)}{\| |v|^2 f_k \|_{L^1}} \leq 0.\ee
Thus, following the proof of Proposition \ref{propstab}, one deduces that there exists a minimizer $f$ of the variational problem  \fref{def-K} with $M_1=\|Q\|_{L^1}$ and $M_j=\|j(Q)\|_{L^1}$ having the same kinetic energy as $Q$, and such that, up to a subsequence,
\be \hat f_k (x,v)=f_k\left(\lambda_k x,\frac{v}{\lambda_k}\right) \rightarrow f \ \textrm{in}\  \mathcal{E}_j,\ \  \textrm{where} \ \ \lambda_k=\left(\frac{ \| |v|^2 Q \|_{L^1}}{ \| |v|^2 f_k \|_{L^1}}\right)^{1/2}.\label{fb-f}\ee
Recall that $Q$ is a steady state of \fref{vm}, thus $J(M_1,M_j)=1$. Since $\int |v|^2f=\int |v|^2Q$, this yields $E_{pot}(f)=E_{pot}(Q)$. Furthermore, from
$$E_{pot}(Q)=E_{pot}(f_k)=\lambda_k^2 E_{pot}(\hat{f}_k)\,\mbox{and}\, E_{pot}(\hat{f}_k)\rightarrow E_{pot}(f)=E_{pot}(Q),$$ we deduce that $$\lambda_k\rightarrow 1\mbox{ as }k\to +\infty.$$
Moreover, we deduce from Theorem \ref{theo1} that $f$ is continuous and satisfies the expression \eqref{expression-Q}. Therefore $\phi_f$ cannot be constant on $\mbox{Supp}(\rho_f)$, which implies
\be
\label{z3}
E_{pot}(Q)=E_{pot}(f)<\| \phi_{f}\|_{L^\infty}\| f \|_{L^1}=\| \phi_{f}\|_{L^\infty}\| Q \|_{L^1}.
\ee
On the other hand, from \fref{cite} and the rescaling inequalities of Appendix \ref{appA}, we get
$$\|\phi_{\hat f_k}\|_{L^\infty}=\lambda_k^2 \frac{E_{pot}(f)}{\|f\|_{L^1}}\to \frac{E_{pot}(f)}{\|f\|_{L^1}}\mbox{ as }k\to +\infty.$$
Hence, since $\phi_{\hat f_k}$ converges to $\phi_f$ in $L^3$, we have
$$\|\phi_f\|_{L^\infty}\leq \lim_{k\to +\infty} \|\phi_{\hat f_k}\|_{L^\infty}=\frac{E_{pot}(f)}{\|f\|_{L^1}},$$
which contradicts the strict inequality \fref{z3}. The proof of Lemma \ref{trouve-bast} is complete.
\qed

\appendix

\section{Rescalings}
\label{appA}

Let $f\in \mathcal{E}_j$ and let $\gamma>0$, $\lambda>0$ and $\mu>0$. Then the rescaled function $\tilde{f}$ defined by $\tilde{f}(x,v)=\gamma f(\frac{x}{\lambda},\mu v)$ satisfies the following identities.

\ms
\ni
{\em Norms}
$$\|\tilde{f}\|_{L^1}=\gamma\frac{\lambda^3}{\mu^3}\|f\|_{L^1}\ \ ;\  \  \|j(\tilde{f})\|_{L^1}=\frac{\lambda^3}{\mu^3}\|j(\gamma f)\|_{L^1}\ \ ; \ \ \left\| |v|^2\tilde{f}\right\|_{L^1}=\gamma\frac{\lambda^3}{\mu^5} \left\| |v|^2 f\right\|_{L^1}.$$

\ni
{\em Functions}
$$\rho_{\tilde{f}}(x)= \frac{\gamma}{\mu^3} \rho_f \left(\frac{x}{\lambda}\right)\ \ ; \    \phi_{\tilde{f}}^P(x)= \gamma\frac{\lambda^2}{\mu^3}\phi_{f}^P\left(\frac{x}{\lambda}\right) \ \ ;  \   \phi_{\tilde{f}}^M(x)= \gamma\frac{\lambda}{\mu^3}\phi_{f}^M\left(\frac{x}{\lambda}\right).$$

\ni
{\em Potential energy}
$$E_{pot}^P(\tilde{f})=\gamma^2\frac{\lambda^5}{\mu^6}E_{pot}^P(f) \ \ ; \  E_{pot}^M(\tilde{f})=\gamma^2\frac{\lambda^4}{\mu^6}E_{pot}^M(f).$$

\begin{lemma}\label{lemresc} Let $f \in \mathcal{E}_j\backslash \{0\}$ and $M_1, M_j>0$. Then there exists an unique pair of positive constants $(\gamma,\lambda)$ such that the rescaled function $\tilde{f}$ defined by \be \tilde{f}(x,v)=\gamma f\left(\frac{\gamma^{1/3}}{\lambda^{1/3}}x,v\right) \label{rescaling} \ee 
satisfies $\|\tilde f\|_{L^1}=M_1$ and $\|j(\tilde f)\|_{L^1}=M_j$. Moreover, $\gamma$ and $\lambda$ satisfy
\be \lambda=\frac{M_1}{\left\|f\right\|_{L^1}} \ \ \textrm{and} \ \ \min(\gamma^{p-1}, \gamma^{q-1} )\leq \frac{M_j\left\|f\right\|_{L^1}}{M_1\left\|j(f)\right\|_{L^1} }\leq \max(\gamma^{p-1}, \gamma^{q-1} ).\label{lambda-gamma}\ee
\end{lemma}
\begin{proof}[Proof. ] The rescaling \eqref{rescaling} gives immediately
$$\|\tilde{f}\|_{L^1}=\lambda\|f\|_{L^1}\quad \mbox{and}\quad \|j(\tilde{f})\|_{L^1}=\frac{\lambda}{\gamma}\|j(\gamma f)\|_{L^1}.$$
Hence, $\tilde f$ satisfies $\|\tilde f\|_{L^1}=M_1$ and $\|j(\tilde f)\|_{L^1}=M_j$ as soon as
$$ \lambda=\frac{M_1}{\left\|f\right\|_{L^1}} \ \ \textrm{and} \ \  \frac{\left\|j(\gamma f)\right\|_{L^1}}{\gamma \left\| j(f)\right\|_{L^1} }= \frac{M_j\left\|f\right\|_{L^1}}{M_1\left\|j(f)\right\|_{L^1} }.$$
The first parameter $\lambda$ is then uniquely determined. Notice also that \eqref{lambda-gamma} is a direct consequence of  the nondichotomy condition \eqref{j-nondicho}. It remains to prove the existence of a unique suitable $\gamma$. 

Consider now the function of $\gamma\in \RR_+^*$ defined by 
$$h(\gamma)= \frac{\left\|j(\gamma f)\right\|_{L^1}}{\gamma \left\| j(f)\right\|_{L^1} }.$$
{}From the nondichotomy condition \eqref{j-nondicho}, we have
$$\lim_{\gamma\to 0} h(\gamma)=0,\qquad \lim_{\gamma\to +\infty}h(\gamma)=+\infty.$$
Moreover, from a direct calculation, one gets
$$h'(\gamma)=\frac{\left\|j'(\gamma f)f\right\|_{L^1}}{\gamma \left\| j(f)\right\|_{L^1} }-\frac{\left\|j(\gamma f)\right\|_{L^1}}{\gamma^2 \left\| j(f)\right\|_{L^1} }\geq (p-1)\frac{\left\|j(\gamma f)\right\|_{L^1}}{\gamma^2 \left\| j(f)\right\|_{L^1} }>0,$$
where we used Assumption (H3) on the function $j$. Hence, there exists a unique $\gamma\in \RR_+^*$ such that
$$h(\gamma)=\frac{M_j \|f\|_{L^1}}{M_1\|j(f)\|_{L^1}}$$
and the Lemma is proved.
\end{proof}

\section{Some properties of radially symmetric potentials}
\label{appC}
\begin{lemma}\label{propC1} There exists a constant $C>0$ such that, for all $f\in \calE_j$ spherically symmetric, we have for all $x\in \RR^3$ 
\be
\label{nono}
 -\frac{C}{|x|}\left\|f\right\|_{L^1}\leq \phi^P_{f}(x)\leq 0.
 \ee
Moreover, for all $0<\alpha<1$, there exists a constant $C_\alpha>0$ such that, for all $f\in \calE_j$ spherically symmetric, we have for all $x\in \RR^3$ 
\be
\label{nono2}
-\frac{C_\alpha}{|x|^{1+\alpha}}\left\|f\right\|_{\calE_j}\leq \phi^M_{f}(x)\leq 0,
 \ee
Recall that $\phi^P_f$ and $\phi^M_f$ are defined by \fref{2pot}.
\end{lemma}
\begin{proof}[Proof. ] 
Passing to the spherical coordinate $s=|y|$ and $x \cdot y=2rs \cos\theta$ in \fref{2pot}, one gets
\be      \phi_{f}^P(x) =- \int_{0}^{+\infty}\int_{0}^\pi \frac{\rho_{f}(s)\sin\theta}{2\left(s^2+r^2-2rs \cos\theta\right)^{1/2}}s^2dsd\theta=- \int_{0}^{+\infty} \rho_{f}(s) g^P_r(s)s^2ds,
\label{ex1}
\ee
\be \phi_{f}^M(x) =- \frac{1}{\pi} \int_{0}^{+\infty}\int_{0}^\pi \frac{\rho_{f}(s)\sin\theta}{s^2+r^2-2rs \cos\theta}s^2dsd\theta=- \frac{1}{\pi}  \int_{0}^{+\infty} \rho_{f}(s) g^M_r(s)s^2ds,
\label{ex2}
\ee
where
$$g^P_r(s)=\frac{\mathbf{1}_{\{s<r\}}(s)}{r}+\frac{\mathbf{1}_{\{s>r\}}(s)}{s} \ \ \textrm{and} \ \ g^M_r(s)=\frac{1}{sr} \ln \left|\frac{r+s}{r-s}\right|.$$
Note that 
$$g^P_r(s)=\frac{1}{r}g^P_{1}\left(\frac{s}{r}\right) \ \ \textrm{and} \ \ g^M_r(s)=\frac{1}{r^2}g^M_{1}\left(\frac{s}{r}\right).$$
Since $g^P_{1}$ belongs to $L^\infty$, \fref{ex1} yields directly \fref{nono}. Next, we remark that $g^M_1$ belongs to $L^k((0,+\infty),s^2ds)$ for all $k\in (3,+\infty)$, which gives
$$\left\| g^M_r  \right\|_{L^k((0,+\infty),s^2ds)}\leq  \frac{C}{r^{2-\frac{3}{k}}}.$$
We finally obtain \fref{nono2} by applying the H\"older inequality to \fref{ex2}. Indeed, thanks to interpolation inequalities and under Assumption (H2), $f\in \calE_j$ implies that $\rho_f\in L^1\cap L^{3/2}((0,+\infty),s^2ds)$. The proof of the lemma is complete.
\end{proof}

\begin{lemma}\label{propC2} Let $(f_n)_{n\geq 1}$ be a bounded sequence of $\mathcal{E}_j$ such that $\rho_{f_n}$ is radially symmetric. Then there exists $f \in \mathcal{E}_j$ such that, up to a subsequence,
$$\left\{ \begin{array} {l}
(i)\ f_n \rightharpoonup f \ \ \textrm{in} \ L^p(\R^6), \\
(ii)\ E_{pot}(f_n)\rightarrow E_{pot}(f), \\
(iii)\ \textrm{for all}\  \frac{3}{2}< q <\frac{3(5p-3)}{4p}, \ \ \phi_{f_n}^M \rightarrow \phi_f^M \ \ \textrm{in}\ L^q(\R^3).
\end{array}\right.$$
\end{lemma}
\begin{proof} Since $p>1$, we have $f_n \rightharpoonup f \ \ \textrm{in} \ L^p(\R^6)$ up to subsequence, which yields {\em (i)}. Let us prove {\em (ii)}. The convergence of the Poisson potential energy is well-known, see e.g. \cite{LMR}. Let us prove the convergence of the Manev potential energy. We remark that
$$E_{pot}^M(f_n)=\| h_{f_n} \|_{L^2}^2 \ \textrm{with} \ h_{f_n}=(-\triangle)^{-1/4}\rho_{f_n}.$$
Hence, from \fref{ineg-inter-M}, we deduce that the sequence $h_{f_n}$ is bounded in $L^2$. Moreover, by interpolation, we have that $\rho_{f_n}$ is bounded in $L^1\cap L^{p_0}$ with $p_0=\frac{5p-3}{3p-1}\in (\frac 32,\frac 53]$ and then,  by standard Sobolev inequalities, the sequence  $((-\triangle)^\varepsilon h_{f_n})$ is bounded in $L^2(\R^3)$ for $\eps>0$ small enough. This yields some local compactness and we have $h_{f_n}\to h_f$ in $L^2_{loc}(\RR^3)$. Hence, to conclude Item {\em (ii)}, it suffices to prove a uniform decay at the infinity. For all $R>0$, we have
\bee
 \left\|h_{f_n}\right\|_{L^2(|x|>R)}&=&\int_{|x|>R}|\phi_{f_n}^M(x)| |\rho_{f_n}(x)|dx\\
 & \leq& C\|\rho_{f_n}\|_{L^{3/2}}\left(\int_{|x|>R}|\phi_{f_n}^M(x)|^3dx\right)^{1/3}\\
 &\leq &C\left(\int_{|x|>R}\frac{1}{|x|^{9/2}}dx\right)^{1/3}=\frac{C'}{R^{3/2}}
 \eee
where we used a H\"older inequality, the uniform boundedness of $\rho_{f_n}$ in $L^{3/2}$ and \fref{nono2} with $\alpha=1/2$.
Finally, we have proved that $h_{f_n}\to h_f$ in $L^2(\RR^3)$, which gives in particular  $E_{pot}^M(f_n)\rightarrow E_{pot}^M(f)$.

The proof of {\em (iii)} is similar. It is sufficient to remark that $\phi_{f_n}^M=(-\triangle)^{-1/2}\rho_{f_n}$ to obtain the local compactness of $(\phi_{f_n}^M)$ in $L^q(\R^3)$ and the uniform decay at the infinity, given by \ref{nono2}, enables to conclude.
\end{proof}

\section{Virial identity}
In this Appendix, we prove the following lemma.
\begin{lemma}
\label{lemvirial}
Let $f\in \calE_j$ be a continuous and compactly supported function which satisfies
\be
\label{eqstatbis}
v\cdot \na_x f-\na_x\phi_f^M\cdot\na_v f=g
\ee
in the distributional sense, where $g$ belongs to $L^1(\RR^6)$. Then the following virial identity holds:
\be
\label{virial}
E_{pot}^M(f)-\int_{\RR^6}|v|^2f(x,v)dxdv=-\int_{\RR^6}(x\cdot v) g(x,v) dxdv.
\ee
\end{lemma}
\begin{proof}
First, integrations by parts give
$$\int_{\RR^6}(x\cdot v) \,(v\cdot \na_xf)\,dxdv=-\int_{\RR^6}|v|^2f(x,v)dxdv$$
and
$$
-\int_{\RR^6}(x\cdot v) \,\na_x \phi_f^M \cdot\na_v f\,dxdv= \int_{\RR^3}\rho_f\,x\cdot\na_x  \phi_f^M\,dx.
$$
Therefore, it remains to prove that this term is well defined and satisfies
\be
\label{manevddd}
\int_{\RR^3}\rho_f\,x\cdot\na_x  \phi_f^M\,dx= E_{pot}^M(f).
\ee
We observe that $\rho_f\in L^1\cap L^\infty(\RR^3)$ since $f$ is continuous and compactly supported. In particular, we have $(-\Delta)^{1/2}\phi^M_f=-\rho_f\in L^2(\RR^3)$. Moreover, from \fref{2pot} we get $\phi^M_f\in L^q(\RR^3)$ for all $q\in ]\frac{3}{2},+\infty]$, in particular $\phi^M_f\in L^2(\RR^3)$. We thus have $\phi_f^M\in H^1(\RR^3)$ and the integral in \fref{manevddd} is well defined.

Let us now regularize the Manev kernel, setting for $\eps>0$
$$ \phi^\eps_{f}(x)=-(-\Delta)^{-1/2-\eps/2}\rho_f=-C_\eps\int_{\R^3}\frac{\rho_f(y)}{|x-y|^{2-\eps}}dy.$$
We have clearly $\phi^\eps_f\to \phi^M_f$ in $H^1(\RR^3)$ as $\eps \to 0$ and then
$$\lim_{\eps\to 0}\int_{\RR^3}\rho_f\,x\cdot\na_x  \phi_f^\eps\,dx=\int_{\RR^3}\rho_f\,x\cdot\na_x  \phi_f^M\,dx.$$
Moreover, we have
\bee
\int_{\RR^3}\rho_f\,x\cdot\na_x\phi^\eps_f\,dx
&=&(2-\eps)C_\eps\int_{\RR^3}\rho_f(x)\rho_f(y)\frac{x\cdot(x-y)}{|x-y|^{4-\eps}}dxdy\\
&=&\frac{2-\eps}{2}C_\eps\int_{\RR^3} \frac{\rho_f(x)\rho_f(y)|x-y|^2}{{|x-y|^{4-\eps}}}dxdy.
\eee
Passing to the limit as $\eps\to 0$ yields \fref{manevddd}. The proof is complete.
\end{proof}

\bs
\bs
\noindent
{\textbf{Acknowledgements.}} We thank Naoufel Ben Abdallah for the proof of \eqref{egal-muphi} from \fref{egal-aphi} and \fref{expression-aphi}. The authors were supported by the french ANR project CBDif.  M. Lemou acknowledges support from the project 'D\'efis \'emergents' funded by the university of Rennes 1 (France). F. M\'ehats also acknowledges support from the french ANR project QUATRAIN and from the INRIA project IPSO.

\end{document}